\documentclass[a4paper,10pt]{article}

\usepackage[margin=3cm,footskip=1cm]{geometry}

\usepackage{amsmath,amssymb,amsthm,bm,mathtools,xspace,booktabs,url}
\usepackage{graphicx,etoolbox}
\usepackage{wrapfig}
\usepackage{ dsfont }
\usepackage{hyperref}
\usepackage{bbold}
\usepackage{tikz}
\usepackage[shortlabels]{enumitem}
\usepackage{xcolor} 
\usepackage{mathrsfs} 
\usepackage{color}

\newcommand{\dd}{\mathrm{d}}

\newcommand{\diag}{\mathrm{diag}}
\makeatletter
\global\let\tikz@ensure@dollar@catcode=\relax
\makeatother
\setlist{
  listparindent=\parindent,
  parsep=0pt,
}
\usepackage{amssymb}

\usepackage{cleveref}

\AtBeginEnvironment{thebibliography}{
  \small
  \setlength\itemsep{0.75em plus 0.5em minus 0.75em}%
}

\usepackage{caption}
\usepackage[raggedright]{titlesec}

\numberwithin{equation}{section}
\usepackage{amsfonts}

\theoremstyle{plain} 
\newtheorem{theorem}{Theorem}[section]
\newtheorem{Lemma}[theorem]{Lemma}
\newtheorem{Proposition}[theorem]{Proposition}

\newtheorem{definition}[theorem]{Definition}

\theoremstyle{definition} 
\newtheorem{example}[theorem]{Example}
\newtheorem{Remark}[theorem]{Remark}
 {
      \theoremstyle{plain}
      \newtheorem{Assumption}{Assumption}
  }

\usepackage{authblk}

\makeatletter
\newcommand\CorrespondingAuthor[1]{
  \begingroup
  \def\@makefnmark{}
  \footnotetext{Corresponding author: #1}
  \endgroup
}
\makeatother

\makeatletter
\renewenvironment{abstract}{%
  \small%
  \begin{center}%
    \bfseries \abstractname\vspace{-.5em}\vspace{\z@}%
  \end{center}%
  \quote%
}
{
\endquote}
\makeatother

\usepackage[above,below,section]{placeins}

\usepackage[all]{onlyamsmath}

\usepackage{multirow}

\usepackage{xcolor}
\usepackage{color}

\usepackage{mathrsfs}
\usepackage{pdfpages}

\newcommand{\CCs}[1]{{\bar{\bm{C}}^{{#1}}}}
\newcommand{\DDs}[1]{{\bar{\bm{D}}^{{#1}}}}

\definecolor{darkmagenta}{rgb}{0.5,0,0.5}
\definecolor{darkgreen}{rgb}{0,0.6,0}
\definecolor{darkblue}{rgb}{0,0,0.6}
\definecolor{darkred}{rgb}{0.8,0,0}
\definecolor{mellow}{rgb}{.847, 0.72, 0.525}

\begin{document}

\title{Duration-dependent stochastic fluid processes and solar energy revenue modeling}

\author{
Hamed Amini \thanks{Department of Industrial and Systems Engineering, University of Florida, Gainesville, FL, USA, email:  aminil@ufl.edu} \ \ \ 
Andreea Minca\thanks{Cornell University, School of Operations Research and Information Engineering, Ithaca, NY, 14850, USA, email: {\tt acm299@cornell.edu}}\ \ \ 
Oscar Peralta\thanks{Cornell University, School of Operations Research and Information Engineering, Ithaca, NY, 14850, USA, email: {\tt op65@cornell.edu}}
}

\maketitle

\begin{abstract}
We endow the classical stochastic fluid process with a duration-dependent Markovian arrival process (DMArP). We show that this provides a flexible model for the revenue of a solar energy generator. In particular, it allows for heavy-tailed interarrival times and for seasonality embedded into the state-space. It generalizes the calendar-time inhomogeneous stochastic fluid process. We provide descriptors of the first return of the revenue process. Our main contribution is based on the uniformization approach, by which we reduce the problem of computing the Laplace transform to the analysis of the process on a stochastic Poissonian grid. Since our process is duration dependent, our construction relies on translating duration form its natural grid to the Poissonian grid. We obtain the Laplace transfrom of the project value based on a novel concept of $n$-bridge and  provide an efficient algorithm for computing the duration-level density of the $n$-bridge. Other descriptors such as the Laplace transform of the ruin process are further provided.
\bigskip

\noindent {\bf Keywords:} Stochastic fluid process; solar energy revenue; first return probability matrix; Laplace transform of project value; ruin time. 

\end{abstract}

\section{Introduction}

With integration of renewable energy sources, the  analysis of the risk and  viability of these projects is critical yet challenging because of their stochastic nature.
We propose a  model for solar energy revenue analysis that  captures the impact of  environmental factors, as well as downtime costs on the long-term revenue generated by the project. Our baseline model is a  stochastic fluid processes (SFPs) endowed with  a novel class of duration-dependent arrival processes. SFPs have played an important role in performance evaluation and risk management in queueing systems \cite{latouche2018analysis}. They have been used to solve a variety of queueing problems, such as determining system response times, calculating waiting times and queue lengths, and analyzing the stability of systems under varying loads. SFPs provide a flexible and tractable framework to model the continuous accumulation of  revenue or rewards over time, with the rewards being positive or negative, depending on the state of some discrete state-space stochastic system. A key  feature is their piecewise linearity between the arrivals, leading to a tractable framework in which one can compute risk descriptors.  

There are a few existing studies, such as \cite{abdelrahman2017markov,deulkar2020sizing, hocaouglu2011stochastic, morf2014sunshine}, that have tackled the problem of modeling solar energy generation with Markov modulation. However, their models rely on time-homogeneous Markov processes, which may not accurately capture  duration-dependent dynamics  influenced by various factors, including weather patterns and down-time. Other recent models, see e.g.,~\cite{huang2021hybrid, kaba2018estimation, blaga2019current}, use machine learning algorithms to better capture complex patterns and relationships between various factors influencing solar energy production. 

We introduce Duration-dependent Markovian Arrival Processes (DMArPs), as an extension of the Markovian Arrival Process (MArP), an  object that is well studied in the queueing theory literature, starting with the seminal paper \cite{neuts1979versatile}. A DMArP is a specific kind of arrival process characterized by a finite state-space jump process with non-homogeneous jump intensities. These intensities restart at random
epochs of time, particularly at the arrival times associated with the DMArP itself. Our framework offers an alternative to the time-homogeneous models introduced in \cite{abdelrahman2017markov,deulkar2020sizing}, enabling  a more realistic representation of complex dynamics that might concurrently exhibit  heavy-tailed interarrival times and seasonality embedded within the state-space. 

\begin{figure}{r}
\captionsetup{justification=raggedright,singlelinecheck=false,labelfont=bf}
\centering
\includegraphics[width=.48\textwidth]{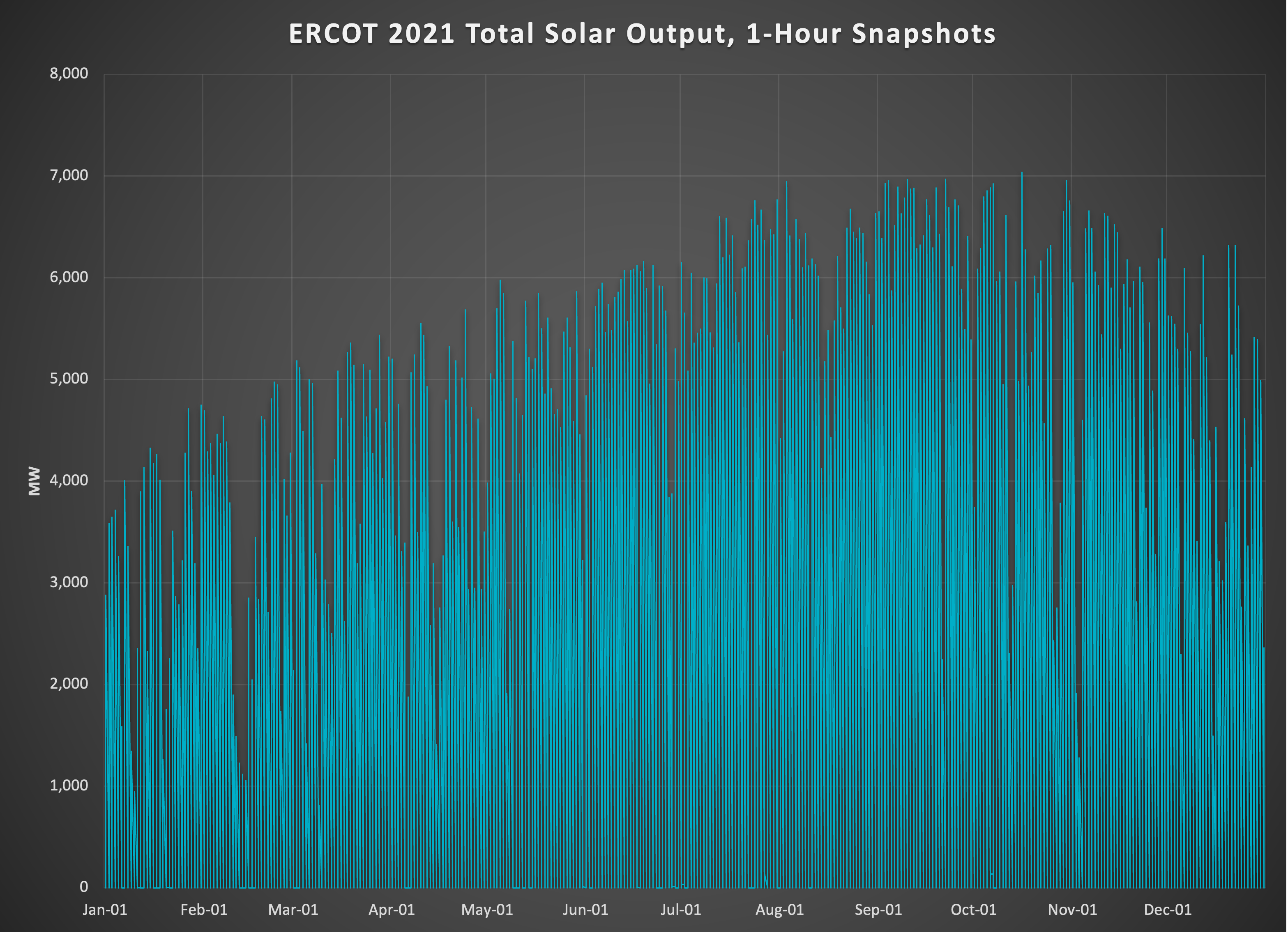}
\caption{Total solar output for the Electric Reliability Council of Texas (ERCOT) Source: \url{https://www.ercot.com/mp/data-products/data-product-details?id=PG7-126-M}}        
\label{fig:lowsolar}
\end{figure}

By using a DMArP, we can model  environmental effects such as  periods of cloudy weather or high temperatures, which  influence the amount of solar energy generated by the project. Some of these states may be extreme in the severity of their impact. 
Moreover, the duration-dependent mechanism in our DMArP allows us to capture situations where the solar energy project experiences downtime, such as maintenance or repairs. These can be directly  related to  environmental factors that cause equipment failures, or indirectly due to the complex interaction of energy demand and price patterns when the system is in a certain state. For example, a storm can trigger high demand coupled with low supply of solar energy and a failure of traditional energy generators due to freezing. 
This was the case of Winter Storm Uri in Texas in February 2021, in which the solar output dropped severely for several consecutive days, see Figure \ref{fig:lowsolar}. 
Simultaneously, numerous traditional generators, which represent the majority of capacity, also failed due to their inability to function at low temperatures, while load demands spiked. Prices, in turn, saw a $20$ fold increase in a matter of hours, leading to losses for an ensemble of generators that had precommitted production levels.
 The financial market plays a key role in the revenue structure and potential losses of these generators; when they produce below the committed level, they are forced to buy energy on the spot market, potentially after a shock in prices; see e.g.,~\cite{woo2011impact, weron2007modeling}. 

The mathematical finance literature on energy production and the associated financial processes remains scarce. In a notable exception, \cite{shrivats2022mean} consider a diffusive model. Their focus is on the optimal production commitment strategy of a generator in the presence of a mean field that captures the  strategies of a large number of other generators. Other recent studies have employed mean field models to address various challenges in the renewable energy sector. For instance, \cite{carmona2022mean, dumitrescu2022energy} explored the use of mean field models to regulate carbon emissions, while \cite{shrivats2022mean} applied these models to study equilibrium pricing in solar renewable energy certificate markets.

In this paper we investigate non-diffusive models, thereby providing an alternative modeling framework.
We do not focus on the exact revenue structure, but we start with a wide variety of given states in which the generator can find itself, as well as revenue rates in each of these states.
We emphasize the parallel to risk modeling in the traditional finance and insurance literature and we seek to  understand the risk of ruin and revenue descriptors for a generator.
Arrivals represent the moments when the system transitions into a state with a significantly negative revenue impact (potentially caused by  extreme weather conditions or equipment failure) or reverts to normal operations, such as through equipment repair or replacement, or the conclusion of adverse weather conditions. By incorporating duration-dependency in all arrival intensities, the model can  accurately capture the long-term behaviour of the project, providing insights into the profitability of the project over time. 

Our primary objective is to develop revenue descriptors and incorporate them into into the analysis of the viability of the solar energy project. We focus on determining a measure related to the depletion of revenue that can be used to compare different policies involving continuous dividend payments and costs associated with each arrival. To achieve this, we adapt the uniformization method for DMArPs and establish a new connection between this approach and the first return times in the context of SFPs. Using this adaptation, we derive exact integral-formulae for the first return descriptors of the associated SFP, which are novel in the literature, even when considering cyclic time-inhomogeneity or Markov-renewal stochastic fluid processes as discussed in previous studies \cite{margolius2016analysis,latouche2004markov}.
Additionally, the model can be used to assess revenue over different time intervals, offering valuable insights into the performance of the project under various policies and system parameters.

Calculating the first return descriptor matrix for the SFP in the solar energy revenue model is useful for several reasons. By analyzing the first return descriptors, we can determine which states are more likely to be revisited by the SFP over time, and therefore which states have a greater impact on the long-term revenue generated by the project. The first return descriptor matrix can also be used to calculate the ruin probability for the solar energy project. By analyzing the ruin probability, we can evaluate the financial risk associated with the project and make informed decisions about risk management strategies. By identifying the factors that contribute to financial loss, such as environmental factors or maintenance requirements, we can develop strategies to mitigate the risk of financial loss and maximize the profitability of the project over the long term.

The paper is structured as follows. 
In Section \ref{motivation} we motivate our stochastic fluid model endowed with duration-dependent arrivals as a potential model for the revenue of a solar energy generator.
The duration-dependent Markovian arrival process is introduced in Section \ref{sec:nonhomogene}.
The descriptors of the first return of the revenue process are provided in Section \ref{sec:exact1}. We analyze the fluid component on stochastic Poissonian grid and establish a new relevant object, the $n$-bridge.
The main inductive approach to obtaining closed form solutions for the $n$-bridge duration-level density is provided in Section~\ref{sec:bridgedensity}.
In Section \ref{sec:ruin} we derive via the Erlangization method a descriptor for the ruin time, namely the Laplace transform of the project value up to the ruin time. 
Section \ref{sec:conclude} gives future directions and concludes.

\section{Motivation and background}
\label{motivation}
Consider a firm investing in a solar energy project. The cumulative revenue generated by the project at time $t\ge 0$, denoted as $F(t)$, is influenced by various environmental factors like weather patterns, temperature, and cloud cover, which show seasonal and intermittent behavior. Our model utilizes a jump process $J=\{J(t)\}_{t\geq0}$ with state space $\mathcal{S}=\{1,\dots,p\}$ to capture these factors.

The jump process $J$ captures exogenous factors  such as weather conditions, solar irradiance levels or energy demand, which impact the firm revenue. Specifically, we model the relationship between the revenue $F=\{F(t)\}_{t\ge 0}$ and $J$ as 
\begin{align}\label{eq:CSFP1}
F(t)=F(0)+\int_0^t r(J(s))\dd s,\quad t\ge 0,
\end{align}
assuming that the firm's net revenue rate at time $t$ is represented by $r(J(t))$, where $r:\mathcal{S}\rightarrow\mathds{R}$ is the instantaneous net revenue rate function. The process $F$ as defined in (\ref{eq:CSFP1}) is known as a stochastic fluid process. In the classical case, see e.g.,  \cite{Rogers:1994uoa,asmussen1995stationary,Karandikar:1995vo}, the underlying jump process $J$ is assumed to be a time-homogeneous Markov process, resulting in the bivariate process $\left\{(F(t),J(t))\right\}_{t\geq 0}$ being a Markov additive process.

In our framework, we are interested in endowing the jump process $J$ with a counting process $N=\{N(t)\}_{t\ge 0}$ that evolves in an intertwined manner with $J$. The counting process keeps track of arrival epochs, which mark the beginning or ending of extreme weather events. We assume that the transition probabilities of the bivariate process $(J,N)$ depend on both the current state of $J$ and the time elapsed since the last arrival in $N$. Clearly, once an extreme event occurs, causing significant deviations in solar energy generation, like a winter storm, the time it takes to return to normal operations cannot be modeled as independent of the elapsed time.

This dependence motivates our duration dependent setup in the next section. Before introducing this extension, let us first  provide preliminaries on the classical time-homogeneous setup for the underlying process $(J,N)$, which reduces to a Markovian arrival process. The classical time-homogeneous Markov jump process $J$ has state space $\mathcal{S}=\{1,\dots,p\}$ and evolves according to an initial distribution $\bm{\alpha}=(\alpha_1,\dots,\alpha_p)$ and an intensity matrix $\bm{Q}=\{q_{ij}\}_{i,j\in \mathcal{S}}$. The matrix $\bm{Q}$ satisfies the conditions that $q_{ij}\ge 0$ for all $j\neq i$, $q_i=-q_{ii}\ge 0$, and $\sum_{i\neq j}q_{ij} = q_i$. The transition probabilities of $J$ after a small period of time $\Delta t>0$ are given by 
\[
\mathds{P}(J(t+\Delta t)=j \;|\; J(t)=i)
=
\left\{ 
\begin{array}{ccc}
q_{ij}\Delta t + o(\Delta t) & \mbox{for}& i\neq j,\, i,j\in\mathcal{S},\\
1-q_{i}\Delta t + o(\Delta t) & \mbox{for}& i\in\mathcal{S}.
\end{array}
\right.
\]

Here, $o(\Delta t)$ denotes an arbitrary real function $g$ such that $g(\Delta t)/\Delta t \rightarrow 0$ as $\Delta t \rightarrow 0$. Additionally, suppose that $\bm{Q}$ admits a decomposition $\bm{Q}=\bm{C} + \bm{D}$, where $\bm{C}=\{ c_{ij}\}_{i,j\in\mathcal{S}}$ and $\bm{D}=\{ d_{ij}\}_{i,j\in\mathcal{S}}$ are such that
\begin{itemize}
\item  $c_{ij}\ge 0$ for all $j\neq i$, $c_i:=-c_{ii}\ge 0$, $\sum_{i\neq j}c_{ij} + \sum_{j}d_{ij} = c_i$,
\item $d_{ij}\ge 0$ for all $i,j$.
\end{itemize}
Adjacent to $J$, the counting process $N=\{ N(t)\}_{t\ge 0}$ over $\mathds{N}_0:=\{ 0,1,2,\dots\}$, evolves interlacingly with $J$ according to the following conditions:
\begin{align*}
&\mathds{P}(N(t+\Delta t) = n, J(t+\Delta t)=j \;\mid\; N(t) = n, J(t)=i )  = \left\{ 
\begin{array}{ccc}
c_{ij}\Delta t + o(\Delta t) & \mbox{for}& i\neq j,\, i,j\in\mathcal{S},\\
1-c_{i}\Delta t + o(\Delta t) & \mbox{for}& i\in\mathcal{S},
\end{array}
\right.,
\end{align*}
\vspace{-0.5cm}
\begin{align*}
&\mathds{P}(N(t+\Delta t) = n+1, J(t+\Delta t)=j \;\mid\; N(t) = n, J(t)=i )  = d_{ij}\Delta t + o(\Delta t).
\end{align*}
The matrix $\bm{C}$ governs the transitions in $J$ that do not lead to a change in the counting process $N$, while $\bm{D}$ drives the transitions in $J$ that result in an arrival in $N$. This distinction enables us to differentiate between  exogenous factors that impact $N$ --called ``extreme"- and those that do not -- called ``regular". It also accounts for the variable durations between consecutive arrivals in $N$. 
While we do not need to impose a special structure on the matrices $\bm{C}$ and $\bm{D}$, it can be helpful to think of a particular example in which   the state space itself is partitioned into "regular" states and ``extreme" states. The net revenue rate could then have completely different values for regular versus extreme states.
The matrix $\bm{C}$ would then be the jumps in the Markov process $J$ between ``regular" states.
Conversely, the matrix $\bm{D}$ leads to transition to/from ``extreme" states. This allows us to capture the effect of weather patterns on the net revenue. Extreme states may lead to highly negative net revenue,  for example, due to equipment failure or large penalties resulting from unmet production levels.

The process $N$ is known as the Markovian arrival process (MArP) with parameters $(\bm{\alpha},\bm{C},\bm{D})$, and we call $J$ its underlying process. The class of MArPs was introduced in \cite{neuts1979versatile} as a method of modeling arrivals in systems that may not exhibit Poissonian behavior, while still maintaining a significant degree of tractability. In particular, it can be shown (see e.g., \cite[Section 10.2.2]{bladt2017matrix}) that the arrival times $S_1,S_2,\dots$ of $N$ are such that the multivariate density function $h_m$ of $(S_1,S_2-S_1,\dots, S_m-S_{m-1})$ takes the form:
\begin{align}
h_m(y_1,y_2,\dots,y_m) = \bm{\alpha} e^{\bm{C} y_1} \bm{D} e^{\bm{C}y_2 } \bm{D} \cdots e^{\bm{C} y_m} \bm{D}\bm{1},\quad y_1,y_2,\dots,y_m\ge 0,\label{eq:transMAP1}
\end{align}
where $e^{(\cdot)}:=\sum_{\ell = 0}^\infty (\cdot)^\ell/\ell!$ denotes the matrix-exponential function, and $\bm{1}$ denotes the column vector of ones of appropriate dimension. Due to the readily implementable matrix nature of their components, MArPs became central to the development of the field known as algorithmic probability, started by Marcel Neuts and collaborators (see e.g., \cite{neuts1994matrix,neuts1995algorithmic}), whose ultimate goal is to provide computationally tractable formulae to describe complex systems. 

In \cite{asmussen1993marked}, it was rigorously established that MArPs are dense within the class of counting processes in $\mathds{R}_+=[0,\infty)$. This means that any arrival behavior can be approximated arbitrarily well by a MArP, demonstrating the flexibility of this approach. In particular, MArPs are well-suited for modeling systems that exhibit regime-switching or dependence between interarrival times, which can be implemented through the underlying process $J$ (see e.g., \cite{artalejo2010markovian}). Below, we present some examples of counting processes that belong to the MArP class.

\begin{example}[Markov-modulated Poisson process]
  Suppose we want to construct a model for a system with multiple arrival regimes, where each regime exhibits Poissonian behavior with a distinct intensity. Specifically, suppose the system switches between regimes according to a Markovian process driven by an intensity matrix $\bm{\Lambda}$. During regime $i$, Poisson arrivals occur at a rate of $v_i\ge 0$. This model is known as a Markov-modulated Poisson process (MMPP) and can be incorporated into the MArP framework by selecting $\bm{D}=\mbox{diag}(v_1, \dots, v_p)$ and $\bm{C}=\bm{\Lambda}-\bm{D}$. MMPPs have been successfully applied in a wide range of queueing models, especially those that may exhibit bursty arrivals, such as telecommunications, manufacturing, and transportation systems (see e.g., \cite{fischer1993markov} and references therein).
\end{example}
\begin{example}[Renewal phase-type process]
Let $Z$ be the termination time of a Markov jump process driven by a subintensity matrix $\bm{T}$ and having initial distribution $\bm{\pi}$. The distribution of $Z$ is said to be phase-type with parameters $(\bm{\pi}, \bm{T})$, a class of distributions that was introduced in \cite{neuts1975probability} and has enjoyed considerable popularity in applied probability due to its tractability and flexibility. Renewal processes with phase-type interarrival times are particularly useful, as several closed-form descriptors are available. In fact, a renewal process with phase-type interarrival times $(\bm{\pi}, \bm{T})$ can be represented as a MArP with parameters $\bm{\alpha}=\bm{\pi}$, $\bm{C}=\bm{T}$ and $\bm{D}=(-\bm{T}\bm{1})\bm{\pi}$. With this choice of parameters, the multivariate density function of the interarrival times $S_1,S_2-S_1,\dots,S_m-S_{m-1}$ of a MArP in (\ref{eq:transMAP1}) takes the form
\begin{align*}
h_m(y_1,y_2,\dots,y_m) & = \bm{\pi} e^{ \bm{T} y_1} \big((-\bm{T}\bm{1})\bm{\pi}\big) e^{\bm{T} y_2} \big((-\bm{T}\bm{1})\bm{\pi}\big) \cdots e^{\bm{T}y_m}\big((-\bm{T}\bm{1})\bm{\pi}\big) \bm{1}\\
& =\left(\bm{\pi} e^{\bm{T} y_1} (-\bm{T}\bm{1})\right)\left(\bm{\pi} e^{\bm{T} y_2} (-\bm{T}\bm{1})\right)\cdots \left(\bm{\pi} e^{\bm{T} y_m } (-\bm{T}\bm{1})\right),
\end{align*}
which corresponds to the product of $m$ independent phase-type densities.
\end{example}

While MArPs have many advantages, they do have limitations when it comes to modeling heavy-tailed interarrival times. By (\ref{eq:transMAP1}), each marginal interarrival time of a MArP belongs to the class of phase-type distributions driven by the subintensity matrix $\bm{C}$. All of these distributions have lighter tails than the exponential distribution of parameter $\lambda_0\in (0, \delta_0)$, where $-\delta_0<0$ is the dominant eigenvalue of $\bm{C}$ \cite[Theorem 4.1.3]{bladt2017matrix}. Therefore, alternative approaches are required to model heavy-tailed interarrival times. In recent years, non-homogeneous methods have been explored to address these needs in the context of univariate and multivariate random variables. For example, \cite{albrecher2019inhomogeneous} proposes a new class of distributions that generalize phase-type distributions by considering the absorption time of a time-inhomogeneous Markov jump process. This idea has been extended to a multidimensional setting in \cite{albrecher2022fitting}. Further analysis in the context of risk management has been conducted in \cite{albrecher2022mortality,bladt2022phase,bladt2022phase2}, where the authors demonstrate the effectiveness of this time-inhomogeneous framework in modelling heavy-tailed phenomena while still maintaining a considerable degree of tractability. In the following section, we show how to integrate these features naturally into the bivariate process $(J,N)$.
 
\section{Duration-dependent Markovian arrival process (DMArP)}\label{sec:nonhomogene}

We define the duration process $U=\{U(s)\}_{s\ge 0}$ associated with a counting process $N$ with arrivals $0=S_0<S_1< S_2< \dots$ as
\[U(s):=s- S_{N(s)},\quad s\ge 0,\]
which represents the elapsed time since the last arrival of $N$. 
There is ample evidence of extreme behaviour for rainfall processes \cite{de2003generalized, salvadori2001generalized, rodriguez1987some}. For example, \cite{de2003generalized} point out that "it is well known that the rainfall process
often features an extreme behavior, which cannot be modeled within an Exponential-like statistical framework" and propose models with Pareto marginals for both the duration and intensity of rainfall process. Critically, duration and intensity are not independent.

To model extreme weather events with a severity that is revenue impacting  and whose duration is heavy tailed,  we consider a non-homogeneous version of the MArP that depends on the process $U$ as follows. Let $N$ have an underlying process $J$ with an initial probability $\bm{\alpha}$, and evolve according to the conditional probabilities
\begin{align}
&\mathds{P}(N(t+\Delta t) = n, J(t+\Delta t)=j \;\mid\; U(t)= u, N(t) = n, J(t)=i )\nonumber\\
&\hspace{6cm}  = \left\{ 
\begin{array}{ccc}
c_{ij}(u)\Delta t + o(\Delta t) & \mbox{for}& i\neq j,\, i,j\in\mathcal{S},\\
1-c_{i}(u)\Delta t + o(\Delta t) & \mbox{for}& i=j, \,i\in\mathcal{S},
\end{array}
\right.\label{eq:defMArPR1}
\end{align}
\vspace{-0.5cm}
\begin{align}
&\mathds{P}(N(t+\Delta t) = n+1, J(t+\Delta t)=j \;\mid\; U(t)= u, N(t) = n, J(t)=i )  = d_{ij}(u)\Delta t + o(\Delta t),\label{eq:defMArPR2}
\end{align}
for all $0\le u \le t$, $n\in\mathds{N}_0$, $i,j\in\mathcal{S}$. The matrices $\bm{C}(\cdot)=\{ c_{ij}(\cdot)\}_{i,j\in\mathcal{S}}$ and $\bm{D}(\cdot)=\{ d_{ij}(\cdot)\}_{i,j\in\mathcal{S}}$ satisfy the following properties:
\begin{itemize}
\item For $s\ge 0$, $c_{ij}(s)\ge 0$ for all $j\neq i$, $c_i(s):=-c_{ii}(s)\ge 0$, $\sum_{i\neq j}c_{ij}(s) + \sum_{j}d_{ij}(s) = c_i(s)$,
\item For $s\ge 0$, $d_{ij}(s)\ge 0$ for all $i,j$,
\item $\bm{C}(\cdot)$ and $\bm{D}(\cdot)$ are right-continuous with left limits.
\end{itemize}
To summarize, when $S_{N(t)}=t-u\ge 0$, $J$ will jump without an accompanying arrival of $N$ occurring in a small time interval after $t$ with intensity $\bm{C}(u)$, while a jump with an attached arrival of $N$ will occur with intensity $\bm{D}(u)$. 
Note that the impact on revenue is now dependent on the duration of the event, and the modeler can posit that transition probabilities to states with high losses increase with the duration of the event. This would amount, for example, to having $d_{ij}(s)$ increasing for  states of successive loss severity. Thereby, one can achieve interdependent duration $S_{n+1} - S_{n}$ and loss $F(S_{n+1}) - F(S_{n})$ with heavy-tailed marginals. 
Moreover, seasonality can also be embedded into the state space and extreme weather events can span multiple state transitions, in particular crossing two seasons. 

With this construction established, we can introduce the arrival process driving the revenue.

\begin{definition}
The process $N$ characterized by (\ref{eq:defMArPR1}) and (\ref{eq:defMArPR2}) is a duration-dependent Markovian arrival process (DMArP) with parameters $(\bm{\alpha},\{ \bm{C}(s)\}_{s\ge 0},\{ \bm{D}(s)\}_{s\ge 0})$, and we refer to $J$ as its underlying process.
\end{definition}

To ensure the existence of a DMArP with parameters $(\bm{\alpha},\{ \bm{C}(s)\}_{s\ge 0},\{ \bm{D}(s)\}_{s\ge 0})$, we initially present a straightforward construction founded on the uniform boundedness assumption described below. A more general construction without this assumption can be found in Appendix \ref{sec:altconst}.

\begin{Assumption}\label{ass:bounded1} There exists $\gamma\in (0,\infty)$ such that $\gamma\ge \sup_{v\ge 0, i\in \mathcal{S}} c_i(v)$.
\end{Assumption}

Under Assumption \ref{ass:bounded1}, we can construct the DMArP $N$ and its underlying process $J$ using uniformization arguments (see e.g., \cite{van2018uniformization}). Let $M=\{ M(t)\}_{t\ge 0}$ be a Poisson process with intensity $\gamma$. For all $s\ge 0$, define
\[\CCs{}(s)=\bm{I}+\frac{1}{\gamma}\bm{C}(s),\quad \DDs{}(s)=\frac{1}{\gamma}\bm{D}(s),\]
where $\bm{I}$ denotes the identity matrix of appropriate dimension. 
The Poisson process $M$ can be viewed as a convenient "stochastic time grid" that is fine enough to capture all events that affect the DMArP process, as shown in Figure~\ref{fig:IMJPR1} for a visual representation.

It can be easily verified that  $\CCs{}(s)=\{ \bar{c}_{ij}(s)\}_{i,j\in\mathcal{S}}$ and $\DDs{}(s)=\{ \bar{d}_{ij}(s)\}_{i,j\in\mathcal{S}}$ are non-negative, and $\CCs{}(s)+ \DDs{}(s)$ is a transition probability matrix. Denote the arrival times of $M$ by $T_0,T_1,T_2,\dots$ with $T_0:=0$. Then we let $N$ and $J$ evolve as follows:
\begin{enumerate}
    \item Let $J(0)\sim \bm{\alpha}$, $N(0)=0$ and $S_0=0$,
    \item For $n=0$, let $J(s)=i_{n}:=J(T_{n})$ and $N(s)=m_{n}:=N(T_{n})$ for all $s\in (T_{n}, T_{n+1})$,
    \item With probability $\bar{c}_{i_{n}, j}(T_{n+1}-S_{m_n})$ let $J(T_{n+1})=j$ and $N(T_{n})=m_n$, and with probability $\bar{d}_{i_{n}, j}(T_{n+1} -S_{m_n})$ let $J(T_{n+1})=j$, $N(T_{n})=m_n+1$ and $S_{m_n+1}=T_{n+1}$.
    \item Repeat steps 2. and 3. for $n=1,2,3,\dots$.
\end{enumerate}

\begin{figure}[h]
  \centering
  \includegraphics[scale=1.7]{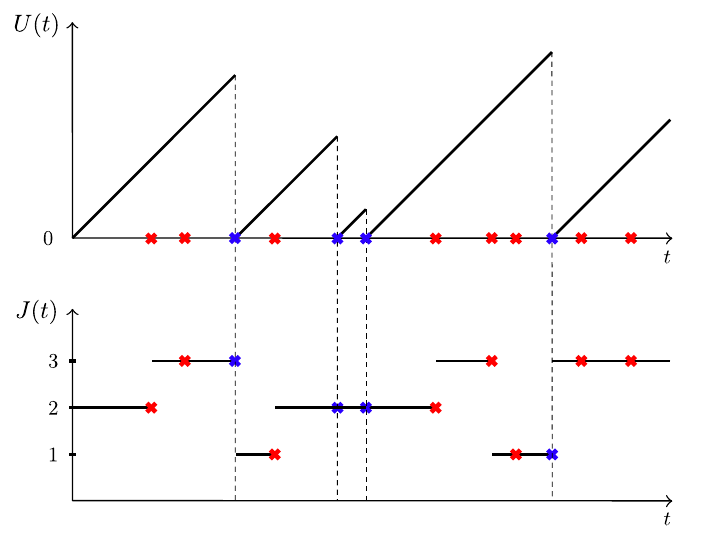}
\caption{(Top) The duration process $U$ is shown, with Poissonian grid $T_1<T_2<\dots$ marked with a cross. At each Poisson arrival, $U$ remains in place with a probability given by $\CCs{}(\cdot)$, event marked in red. A jump of $U$ to $0$ at a given Poisson arrival with a probability given by $\DDs{}(\cdot)$, event marked in blue. (Bottom) The jump process $J$ is shown, with jumps between states of $\mathcal{S}$ occurring at each $T_1,T_2,\dots$ according to the transition probability matrix $\CCs{}(\cdot) + \DDs{}(\cdot)$. It is important to note that at each Poisson arrival, there may or may not be an actual change of state in $\mathcal{S}$. \label{fig:IMJPR1}}
\end{figure}
Using the identity $U(t)=t-S_{N(t)}$ and the right-continuity of $\CCs{}(\cdot)$ and $\DDs{}(\cdot)$, we get
\begin{align*}
&\mathds{P}(N(t+\Delta t) = n, J(t+\Delta t)=j \;\mid\; U(t)= u, N(t) = n, J(t)=i )\\
& = \mathds{P}(M(t+\Delta t)\neq M(t),\, N(t+\Delta t) = n,\, J(t+\Delta t)=j \;\mid\; S_{n}= t-u,\, S_n\le t < S_{n+1},\, J(S_n)=i )\\
&\quad + \mathds{P}(M(t+\Delta t) = M(t),\, N(t+\Delta t) = n,\, J(t+\Delta t)=j \;\mid\; S_{n}= t-u,\, S_n\le t < S_{n+1},\, J(S_n)=i )\\
& = \left[\gamma (\Delta t)\right]\bar{c}_{ij}(u) + \left[1-\gamma(\Delta t)\right]\mathds{1}\{ i=j\} + o(\Delta t) =  \left\{ 
\begin{array}{ccc}
c_{ij}(u)\Delta t + o(\Delta t) & \mbox{for}& i\neq j,\\
1-c_{i}(u)\Delta t + o(\Delta t) & \mbox{for}& i=j,
\end{array}
\right.,\quad\mbox{and}
\end{align*}
\begin{align*}
&\mathds{P}(N(t+\Delta t) = n+1, J(t+\Delta t)=j \;\mid\; U(t)= u, N(t) = n, J(t)=i )\\
& = \mathds{P}(M(t+\Delta t)\neq M(t),\, N(t+\Delta t) = n+1, J(t+\Delta t)=j \;\mid\; S_{n}= t-u,\, S_n\le t < S_{n+1},\, J(S_n)=i )\\
& = \left[\gamma (\Delta t)\right]\bar{d}_{ij}(u) + o(\Delta t) = d_{ij}(u) + o(\Delta t),
\end{align*}
which yields the process $N$ as a DMArP following (\ref{eq:defMArPR1}) and (\ref{eq:defMArPR2}).

It turns out that  survival probabilities for each interarrival time can be expressed in terms of the so-called \emph{product integral}, commonly used to describe non-homogeneous Markovian systems. More generally, the product integral with respect to a measure-valued matrix $\bm{B}(\dd x)$ over the interval $(s,t]$ is defined by
\begin{equation}
\prod_{v\in (s,t]}(\bm{I}+\bm{B}(\dd v)): = \lim_{\max|v_k-v_{k-1}|\rightarrow 0} \prod_{k}\big(\bm{I}+\bm{B}((v_{k-1},v_k])\big),
\end{equation}
where $s=v_0<v_1<\dots<v_n=t$ is a partition of $(s,t]$. Alternatively, $\prod_{v\in (s,t]}(\bm{I}+\bm{B}(\dd v))$ can be defined \cite[Equation (16)]{gill1994lectures} as the unique solution to the integral equation
\begin{equation}\label{eq:volterra1}
\bm{Y}(s,t) = \bm{I} + \int_{(s,t]} \bm{B}(\dd v)\bm{Y}(v,t),\quad 0\le s<t.
\end{equation}
It can be shown \cite[Section 4.4]{gill1990survey} that if $\bm{B}(\dd x)=\bm{Q}(x)\dd x$ for some infinitesimal generator $\bm{Q}(\cdot)$ of a calendar-time non-homogeneous Markov jump process, then $\prod_{v\in (s,t]}(\bm{I}+\bm{B}(\dd v))$ corresponds to the transition probability matrix from time $s$ to $t$. In particular, if $\bm{Q}(\cdot)=\bm{Q}$ (i.e., the associated Markov process is time-homogeneous), then $\prod_{v\in (s,t]}(\bm{I}+\bm{Q}(v)\dd v)=e^{\bm{Q}(t-s)}$. To ease the exposition, below we present a survival probability identity corresponding to $S_1$ in terms of product integrals, with the case $S_{n+1}-S_n$, $n\ge 2$, following in a straightforward manner.
\begin{Lemma}\label{lem:auxprod3}
For $0\le s\le t$ define $\bm{G}(s,t)=\{ g_{ij}(s,t)\}_{i,j\in\mathcal{S}}$ where
\[g_{ij}(s,t)=\mathds{P}(S_1 > t, J(t)=j\,\mid\, S_1>s, J(s)=i).\]
Then,
\begin{equation}\label{eq:Gprodint1}\bm{G}(s,t) = \prod_{v\in (s,t]}(\bm{I}+\bm{C}(v)\dd v),\end{equation}
where the r.h.s. denotes the product integral associated to the measure $\bm{C}(v)\dd v$.
\end{Lemma}

The proof of Lemma~\ref{lem:auxprod3} is provided in Appendix~\ref{sec:lem:auxprod3}.

\begin{theorem}\label{th:multivariateprod1}
Suppose $J(0)\sim\bm{\alpha}=(\alpha_i: i\in\mathcal{S})$, $U(0)=0$, and let $y_1,y_2,\dots, y_n\ge 0$. Then, the multivariate density of ${S_1, S_2-S_1,\dots, S_n-S_{n-1}}$ is given by
\begin{equation}\label{eq:MAPlikedensity}
f_n(y_1, y_2,\dots, y_n)= \bm{\alpha} \bm{G}(y_1) \bm{D}(y_1)\bm{G}(y_2) \bm{D}(y_2) \cdots \bm{G}(y_n) \bm{D}(y_n)\bm{1},
\end{equation}
where 
\[\bm{G}(x):= \prod_{v\in (0,x]} (\bm{I} + \bm{C}(v)\dd v),\quad x\ge 0.\]
Particularly, (\ref{eq:transMAP1}) holds since $\bm{G}(x)=e^{\bm{C}x}$ whenever $\bm{C}(v)=\bm{C}$ for all $v\ge 0$.
\end{theorem}

The proof of Theorem~\ref{th:multivariateprod1} is provided in Appendix~\ref{sec:th:multivariateprod1}.

\medskip

Only a few studies have explored  non-inhomogeneous versions of the Markovian arrival process. Notably, \cite{ledoux2008strong,angius2014approximate,bladt2020matrix} have pursued this topic. However, our approach to non-homogeneity differs from theirs. They consider processes in which jumps are controlled by $\bm{C}(t)$ and $\bm{D}(t)$ at time $t\ge 0$, independently of the time elapsed since the last arrival, a framework that exhibits \emph{calendar-time} non-homogeneity. In contrast, our framework exhibits \emph{interarrival} non-homogeneity, which allows in particular for heavy-tailed interarrival times. Examples \ref{ex:semiMarkov1} and \ref{ex:inhomogeneousPHrenewal} below provide two cases of the DMArP that cannot be replicated in either the calendar-time non-homogeneous MArP framework or the time-homogeneous MArP framework.

\begin{example}[Markov-renewal process]\label{ex:semiMarkov1}
It is well known \cite[Chapter 11]{ross2010introduction} that to simulate any random variable $Z$ with a bounded hazard rate function $h$, one can sample a Poisson process of intensity $\gamma'= \sup_{x\ge 0}\{ h(x)\}$ with arrival epochs $\{ T_n'\}_{n\ge 1}$, a sequence $\{ U_n\}_{n\ge 0}$ of i.i.d. $\mathrm{Unif}(0,1)$ random variables, and set $Z= T_{\sigma_*}'$ where
\[\sigma_*:=\inf\left\{ n\ge 1: U_n \le \frac{h(T_n')}{\gamma'}\right\}.\]
Thus, the restart epochs $\{ S_n\}_{n\ge 0}$ of a DMArP with $\bm{C}(v)=\diag\{ -h_i(v):i\in\mathcal{S}\}$ and $\bm{D}(v)=\{ q_{ij} h_i(v)\}_{i,j\in\mathcal{S}}$, where $\{ h_i\}_{i\in\mathcal{S}}$ is a collection of bounded hazard rate functions and $\bm{Q}=\{ q_{ij}\}_{i,j\in\mathcal{S}}$ is a probability matrix, correspond to the jump times of a \emph{Markov-renewal jump process} \cite{cinlar1969markov}. The latter extends the class of Markov jump processes by allowing for general state-dependent interarrival times, which have survival probability $\exp\left(\int_0^y h_i(s)\dd s\right)$ for $y\ge 0$.
\end{example}

\begin{example}[Time-inhomogeneous phase-type renewal process]  \label{ex:inhomogeneousPHrenewal} \emph{Time-inhomogeneous phase-type distributions} (IPH) were recently introduced in \cite{albrecher2019inhomogeneous} to provide a robust class of univariate distributions on $[0,\infty)$ suitable to model both light and heavy tailed phenomena. To define an IPH distribution, consider a terminating time-inhomogeneous Markov jump process $\left\{Y_t\right\}_{t\ge 0}$ on $\mathcal{S}$ with initial distribution $\bm{\alpha}$ and subintensity generator $\bm{C}(v)$ at time $v\ge 0$. Then the random variable $Z:=\inf\{ t\ge 0 : Y_t = \Delta\}$ is said to follow an IPH distribution of parameters $\big(\bm{\alpha}, \{ \bm{C}(v)\}_{v\ge 0}\big)$ with density function of the form
\[f_{Z}(y)=\bm{\alpha} \bm{G}(y)(-\bm{C}(y)\bm{1}), \quad y\ge 0.\] 
In particular, the arrivals of a DMArP with $\bm{D}(v) := (-\bm{C}(v)\bm{1})\bm{\alpha}$ correspond to the arrival times of a \emph{renewal process} with $\mathrm{IPH}\left(\bm{\alpha},\left\{\bm{C}(v)\right\}{v\ge 0}\right)$ interarrival times. This is verified by noting that, in this case, the multivariate density function of $S_1,S_2-S_1,\dots,S_n-S_{n-1}$ in (\ref{eq:MAPlikedensity}) takes the form
\begin{equation*}
f_n(y_1, y_2,\dots, y_n)= \prod_{\ell=1}^n (\bm{\alpha} \bm{G}(y_\ell) (-\bm{C}(y_\ell)\bm{1})), \quad y_1,\dots, y_n\ge 0.
\end{equation*}
\end{example}

Proposition \ref{prop:marginals1} below shows that, in general, the marginal distribution of each interarrival time $S_n-S_{n-1}$ is IPH-distributed.
\begin{Proposition}\label{prop:marginals1} The arrival times $S_0,S_1,S_2,\dots$ of the DMArP are such that
\begin{equation}\label{MAPlikemarginal}S_n-S_{n-1} \sim \mathrm{IPH}\left(\bm{\alpha}\left(\int_0^\infty \bm{G}(s) \bm{D}(s)\dd s\right)^{n-1}, \{ \bm{C}(v)\}_{v\ge 0}\right),\quad n\ge 1.\end{equation}
\end{Proposition}

The proof of Proposition~\ref{prop:marginals1} is provided in Appendix~\ref{sec:prop:marginals1}. We end this section with the following remark.

\begin{Remark}
Note that the initial vector $\bm{\alpha}\left(\int_0^\infty \bm{G}(s) \bm{D}(s)\dd s\right)^{n-1}$ of the IPH in Proposition \ref{prop:marginals1} may be defective for some $n\ge 1$. In this case, $S_\ell-S_{\ell-1}=\infty$ with positive probability for some $0\le \ell\le n$, meaning that there are at most $n$ extreme events.
\end{Remark}

\section{Descriptors of first return of the revenue process}\label{sec:exact1}
Equipped with the analysis of the driving DMArP, we now turn our attention to the revenue process.
We let $F$ be a SFP with revenue rate function $r:\mathcal{S}\rightarrow \mathds{R}$, driven by  a DMArP $N$, with parameters $(\bm{\alpha},\{ \bm{C}(s)\}_{s\ge 0},\{ \bm{D}(s)\}_{s\ge 0})$ and underlying jump process $J$. We are interested in  the behavior of the processes $F$, $N$, and $J$ up to the \emph{first return time} $\tau$, which is defined as the first time the revenue process falls below initial value
\[\tau  :=\inf\{ s> 0: F(s) < F(0)\}.\]

To facilitate this study, we define $\mathcal{S}^{\pm}$ as the set of states in $\mathcal{S}$ with a positive or negative net revenue rate, respectively, that is,
\begin{align*}
\mathcal{S}^{\pm} & := \{ i\in \mathcal{S} : (\pm 1)r(i) > 0\}.
\end{align*}

Note that if the image of $r$ is contained in $\mathds{R}\setminus\{0\}$, then $\mathcal{S}=\mathcal{S}^+\cup \mathcal{S}^-$. We make this simplifying assumption throughout our analysis, as it is common in the SFP literature. While the general case follows by using similar tools, more care is needed when $0$ rates are allowed (see e.g., \cite[Section 3.1]{asmussen1995stationary}).

In the time-homogeneous case, a common way to describe the behavior of both $F$ and $J$ is to use the first return probability matrix $\bm{\Psi}=\{\psi_{ij}\}_{i\in\mathcal{S}^+, j\in\mathcal{S}^-}$, which is defined as
\[\psi_{ij} := \mathds{P}(\tau<\infty, J(\tau)= j \mid J(0)=i).\]
This matrix can  serve as a performance measure for the system, capturing the occupation probabilities for the environment $J$ at the first return time.

In order to model the profitability of the project from the perspective of the firm in the DMArP-driven SFP framework, we will incorporate additional information about $J$ and $N$ into the descriptor or performance measure.
To achieve that, we consider that the revenue process continuously pays a state-dependent dividend rate. 
More precisely, we model the dividend process via a   mapping  $\sigma: \mathcal{S} \rightarrow \mathds{R}^+$, which satisfies the condition $\sigma(i)=0$ if $i \in \mathcal{S}^-$. By the latter condition, the project only pays dividend in the states with positive revenue rate. 

We also assume that the DMArP process $N$ imposes a cost structure. 
In our motivating example, we think of arrivals of $N$ as ``extreme" weather events. While the fluid process captures the revenue from energy production,  it is natural to assume that the project faces costs at each arrival of the extreme events. These may represent fixed costs associated with each downtime, such as repair costs. They come in addition to the running losses $r$ such as penalties due to unmet production levels.
 We  let $\bm{K} = \{k(i,j)\}_{i,j\in\mathcal{S}}$ be a nonnegative matrix that captures fixed costs, when the project jumps from state $i$ to a state $j$.

Note that in order to use the uniformization technique to compute first return probabilities and associated descriptors of profitability,  we will need to consider the delayed DMArP,  where we fix some $i_0\in\mathcal{S}$ as initial state $J(0)$, but assume that $U(0)=z$ for some $z\ge 0$. This instance occurs whenever an observer arrives at some point in time, and at that instant which is considered the observer's $0$ time, the system's duration process is at exactly $z$ units. In this context, the duration process is defined as
\[U(s):=\left\{\begin{array}{ccc}z+ s & \mbox{if} & 0\le s <S_1\\ s- S_{N(s)} & \mbox{if} & S_1\leq s\end{array}\right. \]
The construction of this delayed case is similar to the non-delayed case:
\begin{enumerate}
    \item Let $J(0)=i_0$, $N(0)=0$, $S_0=0$ and $U(0)=z$,
    \item For $n=0$, let $J(s)=i_{n}:=J(T_{n})$ and $N(s)=m_{n}:=N(T_{n})$ for all $s\in (T_{n}, T_{n+1})$,
    \item With probability $\bar{c}_{i_{n}, j}(T_{n+1}-S_{m_n}+U(S_{m_n}))$ let $J(T_{n+1})=j$ and $N(T_{n})=m_n$, and with probability $\bar{d}_{i_{n}, j}(T_{n+1} -S_{m_n} + U(S_{m_n}))$ let $J(T_{n+1})=j$, $N(T_{n})=m_n+1$ and $S_{m_n+1}=T_{n+1}$.
    \item Repeat steps 2. and 3. for $n=1,2,3,\dots$.
\end{enumerate}

Based on the delayed DMArP, we can capture the profitability of the project using the matrix $\bm{\Psi}^{(z)}(\theta_1,\theta_2)=\{\psi_{ij}^{(z)}(\theta_1,\theta_2)\}_{i\in\mathcal{S}^+, j\in\mathcal{S}^-}$ for $\theta_1,\theta_2, z\ge 0$.  This matrix represents the Laplace transform of the project value, with its elements defined as
\begin{align}
\label{eq:psitheta1}\psi_{ij}^{(z)}(\theta_1,\theta_2)= \mathds{E} \Big(& \exp\big(-\theta_1\int_{0}^\tau \sigma(J(s))\dd s-\theta_2\sum_{\ell: S_\ell<\tau} k(J(S_\ell) -, J(S_\ell))\big)\nonumber\\
&\mathds{1}\{\tau<\infty, J(\tau)= j\} \mid U(0)= z, J(0)=i\Big).
\end{align}
Note that for $z > 0$, the expectation in the r.h.s. of (\ref{eq:psitheta1}) takes into account a process $U$ that is delayed at the beginning by $z$ units of time.
Moreover, the structure is fairly general, and one can ensure that fixed costs are paid only once for each extreme event. For instance, if there is no change of state $J(S_{n})- = J(S_{n})$ then the cost can be set to zero, with large costs occurring when the project jumps from a ``regular" state to an ``extreme" state.
This can be easily extended to the case where the fixed costs occur with each transition of the environmental process, not only at the ``extreme" events. 

By an appropriate choice of dividend and cost functions, the matrix $\bm{\Psi}^{(z)}(\theta_1,\theta_2)$ can describe information about the time spent by $J$ in states with positive revenue rate, and the frequency of jumps that happen at the arrival times of $N$ up to $\tau$. This will add to our understanding of the projects' reliability.

The problem of computing the first return matrix $\bm{\Psi}$ for a time-homogeneous SFP is well-established. Solutions or algorithms, based on the fact that excursions of the level component $F$ are homogeneous with respect to time, have been proposed (see \cite{latouche2018analysis} for a comprehensive overview). Furthermore, extensions to semi-Markov and cyclic time-inhomogeneous stochastic fluid processes have been considered in \cite{latouche2004markov, margolius2016analysis}, which rely on inspecting the additive component at certain regeneration epochs for developing a level-crossing analysis. However, in the case of SFP driven by DMArPs, there are some limitations to this regenerative approach. While the natural regeneration epochs would coincide with the arrivals of the DMArP, note that if $S_1=+\infty$ with positive probability, this regeneration structure may not exist at all. Moreover, if $S_n<+\infty$ for all $n\ge 1$, that would rely on developing a theory for the excursions of the fluid process. 

Therefore, instead of pursuing a regenerative approach at the epochs $S_n$, we will base our analysis on a stochastic Poissonian grid. More specifically, under Assumption \ref{ass:bounded1}, we analyze the fluid component by observing it at the Poisson times $0=T_0<T_1<T_2<\dots$ introduced in the uniformization construction of the DMArP in Section~\ref{sec:nonhomogene}. This relies on the fact that $F$ is linear between $T_1<T_2<\dots$, and hence, $\{ \tau<\infty\}=\{ F(T_n)\le F(0)\mbox{ for some }n\ge 1\}$. In other words, we study the first return of $F$ by analyzing the downcrossing events of $\{ F(T_n)\}_{n\ge 0}$. This idea was first proposed in \cite{bean2019finite}, where the authors constructed an algorithm to compute finite-time return probabilities for time-homogeneous SFPs. In the reminder of this section, we present an algorithm inspired by the ideas in \cite{bean2019finite}, but tailored for the more complex inhomogeneous case. We provide the Laplace transform $\bm{\Psi}^{(z)}(\theta_1,\theta_2)$ for the project value of an SFP driven by a DMArP. Our approach differs from those presented in \cite{latouche2004markov, margolius2016analysis} for studying non-homogeneous SFPs, which rely on translating the duration from the natural grid of extreme events to the stochastic Poisson grid.  This enables us to utilize the full power of the Poissonian grid approach for solving complex excursion problems.

\subsection{Laplace transform of the project value} 
 \label{sec:Laplace}
Let us begin by assuming that $J(0)\in\mathcal{S}^+$. Recall that $0=T_0 < T_1 < T_2 <\cdots$ are Poissonian epochs of intensity $\gamma$, the latter being taken as in Assumption \ref{ass:bounded1}. These epochs include all the jumps and arrivals of $J$ and $N$, respectively. For $n\ge 2$, we define the event
\[\Omega_{n}=\left\{\max\left\{F(T_{0}), F(T_n)\right\}<\min\left\{F(T_1),\dots, F(T_{n-1})\right\}\right\}.\]
We can envision trajectories of $F$ in $\Omega_n$ as those with a ``bridge-like" structure within the time interval $[T_0, T_{n}]$, where the endpoints are at $T_0$ and $T_n$, and all intermediate points are strictly above both $F(T_0)$ and $F(T_n)$. For this reason, we call the paths that belong in $\Omega_n$ as $n$-bridges.
It is important to note that the events $\Omega_{n}$, for $n\ge 2$, are not necessarily disjoint. Figure \ref{fig:Omeganr} provides an example of a sample path of $F$ contained in $\Omega_{n}$ for different choices of $n$.
\begin{figure}[h]
  \centering
  \includegraphics[scale=1.7]{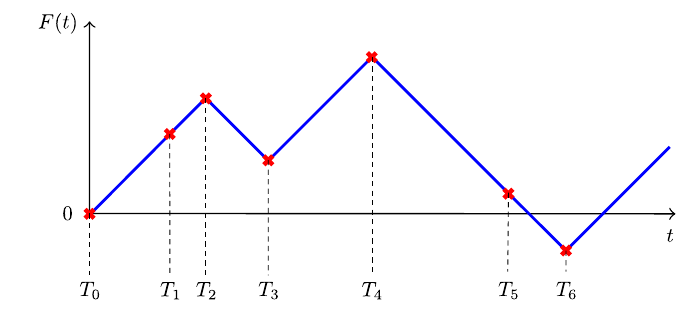}
\caption{A sample path of $F$ contained in $\Omega_{n}$ for $n=3,5,6$. The placement of each red cross represents a Poisson time $T_{k}$ and its respective value in $F$. 
\label{fig:Omeganr}}
\end{figure}

We now introduce the key quantity of this paper, namely the \emph{duration-level density for the $n$-bridge}. This represents the density of $(U(T_n-), F(T_n) - F(0))$ for a given path in $\Omega_n$.
More precisely,  for all $z, s\ge 0$, $\ell\in\mathds{R}$ and $n\ge 2$, let $\bm{\Lambda}^{(n,z)}(\theta_1,\theta_2, s,\ell)=\{  \lambda_{ij}^{(\theta_1,\theta_2, n,z)}(s,\ell)\}_{i\in\mathcal{S}^+,j\in\mathcal{S}^-}$ where
\begin{align*}
\lambda_{ij}^{(n,z)}(\theta_1,\theta_2,s,\ell)& :=\frac{\partial^2 L_{ij}^{(n,z)}(\theta_1,\theta_2,s,\ell)}{\partial s\partial\ell},\\
 L_{ij}^{(n,z)}(\theta_1,\theta_2,s,\ell) &:= \mathds{E}_{i,z}\Bigr(\exp\Bigl(-\theta_1\int_{0}^{T_n }\sigma(J(t))\dd t-\theta_2\sum_{\ell: 0<S_\ell<T_n} k(J(S_\ell) -, J(S_\ell))\Bigr)\\
 &\qquad\qquad\quad \times \mathds{1}\left\{\Omega_{n},\,J(T_{n}-)=j,\, U(T_{n}-)\le s,\, F(T_{n})-F(0)\le \ell\right\}\Bigr),
\end{align*}
and $\mathds{E}_{i,z}$ denotes the expectation under the probability measure $\mathds{P}$ conditioned on $\{ J(0)=i, U(0)=z\}$. In the following we show how the $n$-bridge density  $\bm{\Lambda}^{(n,z)}(s,\ell)$ determines the first return probability matrix $\bm{\Psi}^{(z)}$.

\begin{theorem}\label{th:Psiz3}
For $\theta_1, \theta_2, z\ge 0$,
\begin{equation}\label{eq:appL4}
\bm{\Psi}^{(z)}(\theta_1, \theta_2) = \sum_{n=2}^\infty \bm{L}^{(n,z)}(\theta_1,\theta_2,\infty, 0),
\end{equation}
where
\[\bm{L}^{(n,z)}(\theta_1,\theta_2,\infty, 0):=\int_{s=0}^\infty \int_{\ell=-\infty}^0 \bm{\Lambda}^{(n,z)}(\theta_1,\theta_2,s,\ell)\,\dd \ell \,\dd s.
\]
\end{theorem}

The proof of Theorem \ref{th:Psiz3} is provided in Appendix~\ref{sec:th:Psiz3}.

Theorems \ref{th:Psiz3}  offers a general formula for the first return probability of SFP driven by a DMArP. The formula is expressed in terms of the $n$-bridge densities.
In the next section, we will provide an algorithm to efficiently compute these densities.

\section{Algorithmic approach to computing the $n$-bridge density}
\label{sec:bridgedensity}
In this section, we obtain a closed-form expression for the $n$-bridge duration-level density, denoted as $\bm{\Lambda}^{(n,z)}(\theta_1,\theta_2,s,\ell)$.

The construction is inductive and follows the following steps.

\begin{itemize}

\item Since there is no notion of first return for $n = 1$, the baseline case is the $2$-bridge.
In this instance, the level-duration density  is fully characterized by the joint distribution of the first two interarrival times, if one conditions on the beginning and end state.
In this case, the analysis can be decomposed by conditioning on the event that the intermediate jump was ``extreme" or not, that is, generated by the matrix $\DDs{}$ or $\CCs{}$.  All $2$-bridges are necessarily switching in the intermediate step from a state of positive revenue to a state of negative revenue, i.e., from $\cal S^+$ to $\cal S^-$.

\item For $n> 2$, we consider the minimum intermediate jump of the $n$-bridge, as illustrated in Figure \ref{fig:IntuitionAlg}, and we denote by $T_w$ the time of the minimum intermediate jump. There are three possible situations. 
In a first case, the intermediate minimum corresponds to the first jump. Then the process starting at $T_w = T_1$ is an $(n-1)$-bridge.
In a second case, the intermediate minimum corresponds to the second-to-last jump. Then the process up to $T_w = T_{n-1}$ is an $(n-1)$-bridge.
In the last case, the intermediate minimum is neither the first nor the  ($n-1$)-th jump, and the process is decomposed into two bridges: one up to $T_w$ (a $w$-bridge) and one starting at $T_w$ (a $(n-w)$-bridge).
\end{itemize}

  \begin{figure}[h]
  \centering
  \includegraphics[scale=1.2]{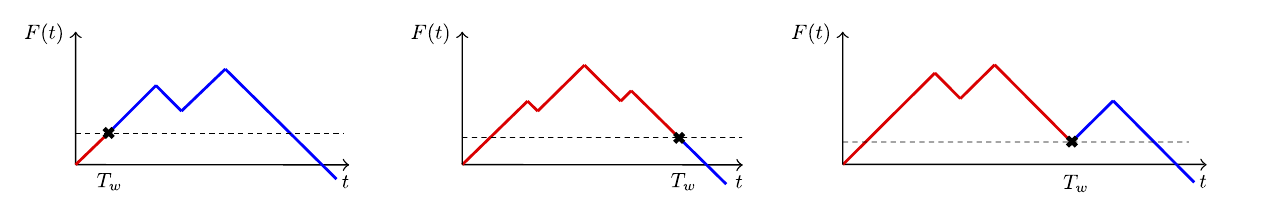}
\caption{The minimum intermediate jump of an $n$-bridge occurs at a time $T_w$ for $w \in \{1, 2, ..., n-1\}$, resulting in three possible decompositions of an $n$-bridge. \label{fig:IntuitionAlg}}
\end{figure}

First us  introduce some notation to simplify our calculations.  For any $v\ge 0$, we partition the matrices $\CCs{}(v)$ and $\DDs{}(v)$ as follows:
\begin{equation}\label{eq:partitionCD1}\CCs{}(v)=\begin{pmatrix}\CCs{++}(v)&\CCs{+-}(v)\\ \CCs{-+}(v)&\CCs{--}(v)\end{pmatrix}\quad\mbox{and}\quad\DDs{}(v)=\begin{pmatrix}\DDs{++}(v)&\DDs{+-}(v)\\ \DDs{-+}(v)&\DDs{--}(v)\end{pmatrix}.\end{equation}
Each block represents the transitions from or to $\mathcal{S}^+$ and $\mathcal{S}^-$. We will also use the Hadamard (or entrywise) multiplication denoted by $\odot$ between any arbitrary matrices $\bm{X}=\left\{ X_{ij}\right\}_{i\in\mathcal{S}^+, j\in\mathcal{S}^-}$ and $\bm{Y}=\left\{ Y_{ij}\right\}_{i\in\mathcal{S}^+, j\in\mathcal{S}^-}$. This operation is defined as $\bm{X}\odot\bm{Y}=\left\{ X_{ij}Y_{ij}\right\}_{i\in\mathcal{S}^+, j\in\mathcal{S}^-}$. Additionally, we define the entrywise composition of the matrix function $\bm{Z}(\cdot)$ with $\bm{X}$ as $\bm{Z}\circ\bm{X}=\left\{ Z_{ij}(X_{ij})\right\}_{i\in\mathcal{S}^+, j\in\mathcal{S}^-}$.

Using the above notation, we can now present a recursive algorithm for computing $\bm{\Lambda}^{(n,z)}(s,\ell)$, where $n$ ranges from $2$ to infinity.
\begin{theorem}\label{th:recursive1}
Fix $\theta_1, \theta_2\ge 0$ and let $\bm{\Lambda}^{(n,z)}(s,\ell)$ be shorthand for $\bm{\Lambda}^{(n,z)}(\theta_1, \theta_2,s,\ell)$. For $z,s\ge 0$ and $\ell\in\mathds{R}$, the collection of matrices $\{ \bm{\Lambda}^{(n,z)}(s,\ell)\}_{n\ge 2}$ can be recursively computed by
\begin{align}
\bm{\Lambda}^{(2,z)}(s,\ell)&=\bm{G}_1^{(z)}(s,\ell)\odot\left(\CCs{+-}\circ \bm{G}_2^{(z)}(s,\ell)\right) + \bm{H}_1^{(z)}(s,\ell)\odot\left(\DDs{+-}\circ \bm{H}_2^{(z)}(s,\ell)\right),\label{eq:Lambdamain1}\\
\bm{\Lambda}^{(n,z)}(s,\ell)&= \sum_{w=1}^{n-1}\bm{\Gamma}^{(n,w,z)}(s,\ell), \quad n\ge 3, \label{eq:Lambdamain2}
\end{align}
where for all $i\in\mathcal{S}^+$ and $j\in\mathcal{S}^-$,
\begin{align*}
\left(\bm{G}_1^{(z)}(s,\ell)\right)_{ij}& = \frac{\gamma^2}{r(i)-r(j)} e^{-\theta_1 \sigma(i)s-\gamma (s-z)}\mathds{1}_{-r(j)(s-z)+\ell\ge 0, r(i)(s-z)-y_2\ge 0},\\
\left(\bm{G}_2^{(z)}(s,\ell)\right)_{ij}& = z+\tfrac{-r(j)(s-z)+\ell}{r(i)-r(j)},\\
\left(\bm{H}_1^{(z)}(s,\ell)\right)_{ij}& = \frac{\gamma^2}{r(i)} \kappa^{+-}_{ij}(\theta_2)\exp\left(-\theta_1 \sigma(i)a_1-\gamma \frac{(r(i)-r(j))(s-z) + \ell}{r(i)}\right)\mathds{1}_{-r(j)(s-z) + \ell\ge 0,\, s-z\ge 0},\\
\left(\bm{H}_2^{(z)}(s,\ell)\right)_{ij}& = z+\frac{-r(j)(s-z) + \ell}{r(i)},
\end{align*}
\begin{align}
\left(\bm{\Gamma}^{(n,1,z)}(s,\ell)\right)_{i\cdot}&= \left(\int_{0}^\infty \gamma e^{-(\gamma+\theta_1\sigma(i)) u} \CCs{++}(u) \bm{\Lambda}^{(n-1,z+u)}(s,\ell-r(i) u)\dd u\right)_{i\cdot}\nonumber\\
&\quad + \left(\int_{0}^\infty \gamma e^{-(\gamma+\theta_1\sigma(i)) u} \left(\bm{\kappa}^{++}(\theta_2)\odot\DDs{++}(u)\right) \bm{\Lambda}^{(n-1,0)}(s,\ell-r(i) u)\dd u\right)_{i\cdot},\label{eq:Gamma1}\\
\bm{\Gamma}^{(n,w,z)}(s,\ell)&= \int_0^\infty \int_{v=0\vee \ell}^\infty \bm{\Lambda}^{(w,z)}(u, v)\bar{\bm{C}}^{-+}(u) \bm{\Lambda}^{(n-w,u)}(s, \ell-v) \dd u \dd v \nonumber\\
 + \int_0^\infty \int_{v=0\vee \ell}^\infty & \bm{\Lambda}^{(w,z)}(u, v)\left(\bm{\kappa}^{-+}(\theta_2)\odot\bar{\bm{D}}^{-+}(u)\right) \bm{\Lambda}^{(n-w,0)}(s, \ell-v)\dd u \dd v, \quad 2 \le w \le n-1,\label{eq:Gamma2}\\
\left(\bm{\Gamma}^{(n,n-1,z)}(s,\ell)\right)_{\cdot j} & = \left(\int_0^s \gamma e^{-\gamma u} \bm{\Lambda}^{(n-1,z)}(s-u,\ell-r(j) u) \bar{\bm{C}}^{--}(u)\dd u\right)_{\cdot j} \nonumber\\
&\quad + \left(\gamma e^{-\gamma s}\int_0^\infty  \bm{\Lambda}^{(n-1,z)}(u,\ell-r(j) s) \left(\bm{\kappa}^{--}(\theta_2)\odot\bar{\bm{D}}^{--}(u)\right)\dd u\right)_{\cdot j},\label{eq:Gamma3}
\end{align}
with 
\[\left\{e^{-\theta_2 k(i,j)}\right\}_{i,j\in\mathcal{S}}=:\begin{pmatrix}\bm{\kappa}^{++}(\theta_2) & \bm{\kappa}^{+-}(\theta_2) \\ \bm{\kappa}^{-+}(\theta_2) & \bm{\kappa}^{--}(\theta_2)\end{pmatrix},\]
$(\bm{A})_{i\cdot}$ denoting the $i$-th row of $\bm{A}$, and $(\bm{A})_{\cdot j}$ denoting the $j$-th column of $\bm{A}$.
\end{theorem}

The proof of Theorem~\ref{th:recursive1} is provided in Appendix~\ref{sec:th:recursive1}.

The first step of the algorithm treats the baseline case of the $n$-bridge for $n=2$. The inductive step considers the minimum intermediate jump of an $n$-bridge for $n > 2$, as shown in Figure \ref{fig:IntuitionAlg}. Based on the occurrence of this minimum intermediate jump at a time $T_w$, $w = 1,2,..., n-1$, the intermediate $n$-bridge is decomposed into a $w$-bridge and an $(n-w)$-bridge.

\section{Ruin of the revenue process}
\label{sec:ruin}
In the previous sections, we demonstrated how to analyze the DMArP  on the Poissonian stochastic grid. We can now use the the first return probabilities and randomization of the initial level $F(0)$ to derive ruin time descriptors.
In this section, we derive a descriptor for the \emph{ruin time} $\rho$ of the DMArP-driven SFP $F$, with the initial condition $F(0)=u\ge 0$, where
\[\rho=\inf\{s>0 : F(s)< 0\}.\]
In existing time-homogeneous SFP literature, the ruin time $\rho$ is investigated using successive downcrossing arguments (see e.g., \cite[Section 3]{asmussen1995stationary}). However, these arguments rely heavily on the time-homogeneous property of the classic SFP, which does not hold when a DMArP is used to model the process.

As an alternative, we propose an approximation method based on randomization. Specifically, we study fixed stochastic systems at a specified point by substituting that point with a random variable that closely approximates it. To be more specific, let us define
\begin{align}\label{eq:rand1}&\psi_{ij}^{(z)}(\theta_1,\theta_2,u) :=\nonumber \\
&\mathds{E}\left(\left.
\begin{array}{c} \exp\left(-\theta_1\int_{0}^\rho \sigma(J(s))\dd s-\theta_2\sum_{\ell: S_\ell<\rho} k(J(S_\ell) -, J(S_\ell))\right)\\
\times \mathds{1}\{\rho<\infty, J(\rho)= j\}
\end{array}
\right|\, U(0)= z, J(0)=i, F(0)=u\right),\end{align}
which corresponds to the ruin descriptor of process $F$ with the initial state $F(0)=u$. The randomization methods for our current context suggest replacing $u$ in (\ref{eq:rand1}) with a random approximation, denoted $u_*$, resulting in
\begin{equation}\label{eq:rand2}
\psi_{ij}^{(z)}(\theta_1,\theta_2,u)\approx\mathds{E}(\psi_{ij}^{(z)}(\theta_1,\theta_2,u_*)).\end{equation}
A suitable choice for $u_*$ is an Erlang random variable with shape $n$ and rate $n/u$, ensuring that $\mathds{E}(u_*)=u$ and $\mbox{Var}(u_*)=u^2/n$. As $n$ approaches infinity, $u_*$ converges in probability to $u$, suggesting that approximation \ref{eq:rand2} is accurate for large values of $n$. This specific technique, called Erlangization, is widely utilized in stochastic modeling due to its tractability (see, for example, \cite{carr1998randomization1,asmussen2002erlangian1,bladt2019parisian1,albrecher2022randomized}).

In our situation, we can utilize the Erlangization method to conveniently approximate the ruin descriptor $\psi_{ij}^{(z)}(\theta_1,\theta_2,u)$. To simplify the setup, assume $z=0$ and $J(0)=i_0\in\mathcal{S}^+$. We then explore an alternative SFP, denoted as $F_*=\{F_*(t)\}_{t\ge 0}$, which is driven by a DMArP $J_*=\{J_*(t)\}_{t\ge 0}$ and has state space $\mathcal{S}_*=\mathcal{S}_*^+\cup\mathcal{S}_*^-$ and parameters $(\bm{\alpha}_*,\{ \bm{C}_*(s)\}_{s\ge 0},\{ \bm{D}_*(s)\}_{s\ge 0})$. The idea is to embed the original $F$ with the initial state $F(0)=u_*$ into $F_*$ with $F_*(0)=0$ in a time-shifted fashion. This is accomplished by allowing $F_*$ to increase linearly at a rate of $1$ up to time $u_*$ and concatenating the path of $F$ afterward. Since the Erlang random variable $u_*$ can be interpreted as a sum of $n$ exponential random variables with a rate of $n/u$, this idea can be implemented by artificially allowing $J_*$ to visit $n$ sequential states with an exit rate of $n/u$ before concatenating the original $J$ (see Figure \ref{fig:Erlang}). 
\begin{figure}[h]
  \centering
  \includegraphics[scale=1.15]{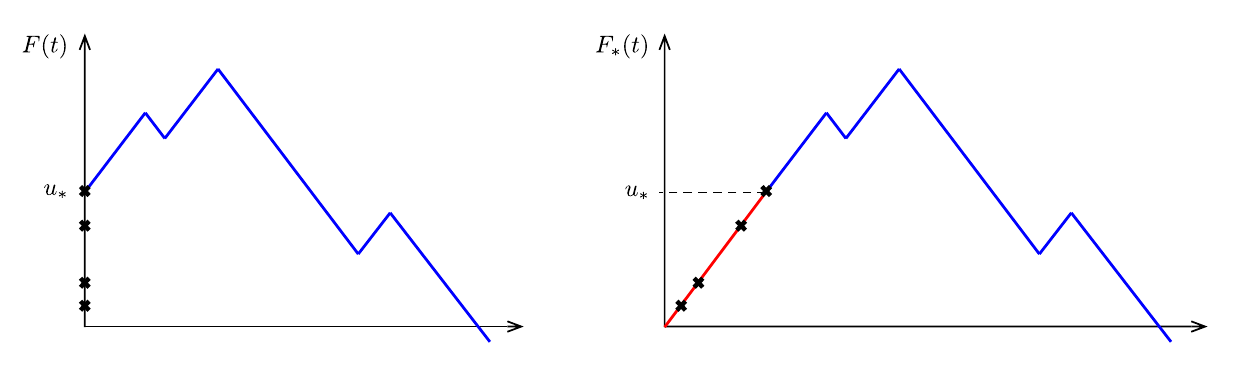}
\caption{Embedding of the process $F$ started at $u$ into the process $F_*$ started at $0$ with added artificial states for $n=4$. Jumps within the artificial states are shown with a cross, resulting in a height $u_*$ which is Erlang-distributed of shape $4$ and rate $4/u$.\label{fig:Erlang}}
\end{figure}
The resulting SFP is then obtained by selecting appropriate parameters and transition probabilities to match the original process $F$, which is done as follows:

\begin{itemize}
\item $\mathcal{S}_*^+=\{1_E,2_E,\dots, n_E\}\cup \mathcal{S}^+$ and $\mathcal{S}_*^-= S^+$;
\item $\bm{\alpha}_*$ is the probability vector $(1,0,0,\dots, 0)$;
\item for all $v\ge 0$, $\bm{C}_*(v)$ is block partitioned in $\mathcal{S}_*^+\cup \mathcal{S}_*^-$ as
\[\left(\begin{array}{c|c}\bm{C}^{++}_*(v) & \bm{C}^{+-}_*(v)\\
\hline
\bm{C}^{-+}_*(v)& \bm{C}^{--}_*(v)\end{array}\right)=\left(\begin{array}{cc|c}-(n/u)\bm{I}& \bm{0} & \bm{0}\\
\bm{0}& \bm{C}^{++}(v) & \bm{C}^{+-}(v)\\
\hline
\bm{0}&\bm{C}^{-+}(v)& \bm{C}^{--}(v)\end{array}\right);\]
\item for all $v\ge 0$, $\bm{D}_*(v)$ is block partitioned in $\mathcal{S}_*^+\cup \mathcal{S}_*^-$ as
\[\left(\begin{array}{c|c}\bm{D}^{++}_*(v) & \bm{D}^{+-}_*(v)\\
\hline
\bm{D}^{-+}_*(v)& \bm{D}^{--}_*(v)\end{array}\right)=\left(\begin{array}{cc|c}(n/u)\bm{E}_1 & (n/u)\bm{E}_2 & \bm{0}\\
\bm{0}& \bm{D}^{++}(v) & \bm{D}^{+-}(v)\\
\hline
\bm{0}&\bm{D}^{-+}(v)& \bm{D}^{--}(v)\end{array}\right),\]
where $\bm{E}_1$ represents the upper diagonal matrix of ones, and $\bm{E}_2$ the matrix which is equal to one in its $(n,i_0)$ entry and zero elsewhere;
\item the net revenue rate function $r_*(\cdot)$ is defined by $r_*(i)=1$ for all $i\in\{1_E,2_E,\dots, n_E\}$, and $r_*(j)=r(j)$ for all $j\in \mathcal{S}^+\cup\mathcal{S}^-$.
\end{itemize}
Consequently, $\psi_{i_0,j}^{(0)}(\theta_1,\theta_2,u)$ could be approximated by the first return descriptor of $F_*$ which takes the form
\begin{align}\label{eq:rand1}&\psi_{1_E,j}^{*(0)}(\theta_1,\theta_2)=\nonumber \\
&\mathds{E}\left(\left.
\begin{array}{c} \exp\left( -\theta_1\int_{0}^{\tau_*} \sigma(J_*(s))\dd s - \theta_2\sum_{\ell: S_{*\ell}<\tau_*} k(J_*(S_{*\ell}) -, J_*(S_{*\ell})) \right)\\
\times \mathds{1}\{\tau_*<\infty, J_*(\tau_*)= j\}
\end{array}
\right|\, U_*(0)= z, J_*(0)=1_E\right),\end{align}
where $S_{*\ell}$, $U_*$, and $\tau_*$ correspond to the interarrival times, duration process, and first return time of the SFP $F_*$. Here, the functions $\sigma$ and $k$ are understood as extensions of their non-randomized SFP versions and are equal to zero whenever their arguments are outside of $\mathcal{S}^+\cup\mathcal{S}^-$.

\subsection{Calendar-time dependence}

In the case where the modeler is only concerned with calendar time-inhomogeneity, and not with duration dependence, it is possible to set $\bm{D}(\cdot)=\bm{0}$ in our DMArP model. This results in the underlying process $J$ being a calendar-time non-homogeneous Markov jump process. Incorporating exact seasonality in the transition rates is possible in this case, as opposed to embedding the seasonality into the state-space. However, sacrificing the duration dependence is a trade-off. For the calendar-time inhomogeneous case, we can obtain additional results, namely the finite time return probabilities.

\begin{theorem}\label{th:timeinhomogeneouspsi1}
Let $F$ be an SFP driven by a DMArP $(\{ \bm{C}(v)\}_{v\ge 0}, \{ \bm{D}(v)\}_{v\ge 0})$ with $\bm{D}(v)=\bm{0}$ for all $v\ge 0$. For $t> 0$, let $\bm{\Psi}^{(z,*)}(t)=\{ \Psi_{ij}^{(z,*)}(t)\}_{i\in\mathcal{S}^+,j\in\mathcal{S}^-}$ be the \emph{finite time return probability matrix over $T_1<T_2<\dots$} defined by
\[\Psi_{ij}^{(z,*)}(t)=\mathds{P}_{i,z}(\tau_*\le t, J(\tau_*-)=j),\]
where
\[\tau_* =\inf \{ T_n: F(T_n)< F(0)\}.\]
Then,
\begin{equation}\label{eq:Psistar3}
\bm{\Psi}^{(z,*)}(t) = \sum_{n=2}^\infty \bm{L}^{(n,z)}(0,0, t+z, 0).
\end{equation}
Furthermore, for fixed $t>0$,
\begin{align}\label{eq:ratetruncation1}
\sum_{n=m}^\infty \big(\bm{L}^{(n,z)}(0,0,t+z, 0)\big)_{ij} = o\left(e^{-m(\log m-\log(\gamma t) - 1)}\right)\quad\mbox{as}\quad m\rightarrow\infty.
\end{align}
\end{theorem}

The proof of Theorem~\ref{th:timeinhomogeneouspsi1} is provided in Appendix~\ref{sec:th:timeinhomogeneouspsi1}.

In this case, similar to first return descriptors presented in Section~\ref{sec:Laplace}, Theorem \ref{th:timeinhomogeneouspsi1} establishes the finite time first return probabilities as an infinite series.
In addition, it establishes a rate of convergence for this series. 
This result allows us to obtain the probability of ruin before a given time, as well as the probability of being ruined in a given state.

\section{Concluding remarks}\label{sec:conclude}

We developed a risk assessment framework for a project that is driven by a duration-dependent Markov arrival process, a novel process that can account for heavy-tailed interarrival times. At the same time, seasonality and extreme events can be  embedded into the structure of the state-space itself.
By treating the revenue process as a risk process, we derive indicators of long term viability as well as ruin probabilities.

We obtain the Laplace transform of the project value based on the analysis of a new construction, namely the $n$-bridge: the revenue process  restricted to paths with $n$ arrivals and  for which it remains above the beginning and the end points at all intermediate arrival times. We provide an algorithmic approach to computing the $n$-bridge duration-level density, which, in turn, provides a tractable way to compute the Laplace transform of the project value as well as the first return probabilities and ruin probabilities.

There are several practical implications of our results.
The first return descriptor matrix allows us to rank the state-space in terms of the probabilities of visiting those states at the first return time.
This would inform decision makers on the potential costs of resuming regular operations, since the Markovian system would continue from those states. There are likely costs or capital injections required to bring the revenue process above a certain barrier or to allow it to linger below its initial value for a prescribed time, such as a Parisian-type ruin. These type of problems pertain to the management of the process around the first return time and are left for future research.

By understanding the risk of ruin and its dependence on the fixed and continuous cost structure, stakeholders can better assess the risks associated with solar energy projects under different dividend and cost structures. 
This information as well as the risk assigned to various states allows us to identify potential threats to revenue generation and viability.

In our work, the transition matrices are given. Of course, in reality, these matrices are subject to change. Identifying potential revenue shortfalls in advance would lead the decision-maker to plan maintenance and repair in such a way that the system's transitions avoid certain states.

 As a future direction, one can define a control problem in which the intensity matrices $\bm{C}$ and $\bm{D}$ are policy-dependent in addition to being duration dependent, and where the fixed costs also depend on the action taken by the controller. The controller would then seek to maximize profit, and it is understood that the controller's policy impacts the transitions as well as the profit structure. At any point in time, interventions can be made to help minimize the average future duration of negative revenue impact, thus ensuring the project remains profitable in the long run. Endowed with the optimal policy and an understanding of associated costs, decision-makers can make informed choices regarding project design, resource allocation, and mitigation strategies.

The insights gained from the analysis of the ruin time can assist investors and policymakers in managing the risks associated with solar energy projects. For example, by understanding the potential shortfall, investors can purchase insurance. Policymakers can also use the analysis of the ruin time to design policies and regulations that promote the sustainability and profitability of solar energy systems.

\paragraph{Acknowledgements.} The authors gratefully acknowledge financial support from the National Science Foundation through grant \#1653354. Additionally, OP acknowledges financial support from the Swiss National Science Foundation Project 200021\_191984.

\bibliographystyle{abbrv}
\bibliography{oscar}

\begin{thebibliography}{10}

\bibitem{abdelrahman2017markov}
O.~H. Abdelrahman.
\newblock A {M}arkov-modulated diffusion model for energy harvesting sensor
  nodes.
\newblock {\em Probability in the Engineering and Informational Sciences},
  31(4):505--515, 2017.

\bibitem{albrecher2022randomized}
H.~Albrecher and J.~C. Araujo-Acuna.
\newblock On the randomized {S}chmitter problem.
\newblock {\em Methodology and Computing in Applied Probability},
  24(2):515--535, 2022.

\bibitem{albrecher2019inhomogeneous}
H.~Albrecher and M.~Bladt.
\newblock Inhomogeneous phase-type distributions and heavy tails.
\newblock {\em Journal of Applied Probability}, 56(4):1044--1064, 2019.

\bibitem{albrecher2022mortality}
H.~Albrecher, M.~Bladt, M.~Bladt, and J.~Yslas.
\newblock Mortality modeling and regression with matrix distributions.
\newblock {\em Insurance: Mathematics and Economics}, 107:68--87, 2022.

\bibitem{albrecher2022fitting}
H.~Albrecher, M.~Bladt, and J.~Yslas.
\newblock Fitting inhomogeneous phase-type distributions to data: the
  univariate and the multivariate case.
\newblock {\em Scandinavian Journal of Statistics}, 49(1):44--77, 2022.

\bibitem{angius2014approximate}
A.~Angius and A.~Horv{\'a}th.
\newblock Approximate transient analysis of queuing networks by decomposition
  based on time-inhomogeneous {M}arkov arrival processes.
\newblock In {\em Proceedings of the 8th International Conference on
  Performance Evaluation Methodologies and Tools}, pages 255--262, 2014.

\bibitem{artalejo2010markovian}
J.~R. Artalejo and A.~G{\'o}mez-Corral.
\newblock Markovian arrivals in stochastic modelling: a survey and some new
  results.
\newblock {\em Statistics and Operations Research Transactions}, pages
  101--156, 2010.

\bibitem{asmussen1995stationary}
S.~Asmussen.
\newblock Stationary distributions for fluid flow models with or without
  {B}rownian noise.
\newblock {\em Communications in Statistics. Stochastic Models}, 11(1):21--49,
  1995.

\bibitem{asmussen2002erlangian1}
S.~Asmussen, F.~Avram, and M.~Usabel.
\newblock Erlangian approximations for finite-horizon ruin probabilities.
\newblock {\em ASTIN Bulletin: The Journal of the IAA}, 32(2):267--281, 2002.

\bibitem{asmussen1993marked}
S.~Asmussen and G.~Koole.
\newblock Marked-point processes as limits of {M}arkovian arrival streams.
\newblock {\em Journal of Applied Probability}, 30(2):365--372, 1993.

\bibitem{bean2019finite}
N.~Bean, G.~Ngyuen, and F.~Poloni.
\newblock A new algorithm for time-dependent first-return probabilities of a
  fluid queue.
\newblock {\em Matrix-Analytic Methods in Stochastic Models}, pages 18--22,
  2019.

\bibitem{bladt2022phase}
M.~Bladt.
\newblock Phase-type distributions for claim severity regression modeling.
\newblock {\em ASTIN Bulletin: The Journal of the IAA}, 52(2):417--448, 2022.

\bibitem{bladt2020matrix}
M.~Bladt, S.~Asmussen, and M.~Steffensen.
\newblock Matrix representations of life insurance payments.
\newblock {\em European Actuarial Journal}, 10:29--67, 2020.

\bibitem{bladt2017matrix}
M.~Bladt and B.~F. Nielsen.
\newblock {\em Matrix-Exponential Distributions in Applied Probability},
  volume~81.
\newblock Springer, 2017.

\bibitem{bladt2019parisian1}
M.~Bladt, B.~F. Nielsen, and O.~Peralta.
\newblock Parisian types of ruin probabilities for a class of dependent
  risk-reserve processes.
\newblock {\em Scandinavian Actuarial Journal}, 2019(1):32--61, 2019.

\bibitem{bladt2022phase2}
M.~Bladt and J.~Yslas.
\newblock Phase-type mixture-of-experts regression for loss severities.
\newblock {\em Scandinavian Actuarial Journal}, pages 1--27, 2022.

\bibitem{blaga2019current}
R.~Blaga, A.~Sabadus, N.~Stefu, C.~Dughir, M.~Paulescu, and V.~Badescu.
\newblock A current perspective on the accuracy of incoming solar energy
  forecasting.
\newblock {\em Progress in energy and combustion science}, 70:119--144, 2019.

\bibitem{carmona2022mean}
R.~Carmona, G.~Dayan{\i}kl{\i}, and M.~Lauri{\`e}re.
\newblock Mean field models to regulate carbon emissions in electricity
  production.
\newblock {\em Dynamic Games and Applications}, 12(3):897--928, 2022.

\bibitem{carr1998randomization1}
P.~Carr.
\newblock Randomization and the {A}merican put.
\newblock {\em The Review of Financial Studies}, 11(3):597--626, 1998.

\bibitem{cinlar1969markov}
E.~Cinlar.
\newblock Markov renewal theory.
\newblock {\em Advances in Applied Probability}, 1(2):123--187, 1969.

\bibitem{davis1984piecewise}
M.~H. Davis.
\newblock Piecewise-deterministic {M}arkov processes: {A} general class of
  non-diffusion stochastic models.
\newblock {\em Journal of the Royal Statistical Society: Series B
  (Methodological)}, 46(3):353--376, 1984.

\bibitem{davis2018markov}
M.~H. Davis.
\newblock {\em Markov models and Optimization}.
\newblock Routledge, 2018.

\bibitem{de2003generalized}
C.~De~Michele and G.~Salvadori.
\newblock A generalized {P}areto intensity-duration model of storm rainfall
  exploiting 2-copulas.
\newblock {\em Journal of Geophysical Research: Atmospheres}, 108(D2), 2003.

\bibitem{deulkar2020sizing}
V.~Deulkar, J.~Nair, and A.~A. Kulkarni.
\newblock Sizing storage for reliable renewable integration: {A} large
  deviations approach.
\newblock {\em Journal of Energy Storage}, 30:101443, 2020.

\bibitem{dumitrescu2022energy}
R.~Dumitrescu, M.~Leutscher, and P.~Tankov.
\newblock Energy transition under scenario uncertainty: a mean-field game
  approach.
\newblock {\em arXiv preprint arXiv:2210.03554}, 2022.

\bibitem{feller1968introduction}
W.~Feller.
\newblock {\em An Introduction to Probability Theory and Its Applications:
  Volume I}, volume~1.
\newblock John Wiley \& Sons, 1968.

\bibitem{fischer1993markov}
W.~Fischer and K.~Meier-Hellstern.
\newblock The {M}arkov-modulated {P}oisson process (mmpp) cookbook.
\newblock {\em Performance Evaluation}, 18(2):149--171, 1993.

\bibitem{gill1994lectures}
R.~D. Gill.
\newblock Lectures on survival analysis.
\newblock In {\em Lectures on Probability Theory}, pages 115--241. Springer,
  1994.

\bibitem{gill1990survey}
R.~D. Gill and S.~Johansen.
\newblock A survey of product-integration with a view toward application in
  survival analysis.
\newblock {\em The Annals of Statistics}, pages 1501--1555, 1990.

\bibitem{glynn1987upper}
P.~W. Glynn.
\newblock Upper bounds on {P}oisson tail probabilities.
\newblock {\em Operations Research Letters}, 6(1):9--14, 1987.

\bibitem{hocaouglu2011stochastic}
F.~O. Hocao{\u{g}}lu.
\newblock Stochastic approach for daily solar radiation modeling.
\newblock {\em Solar Energy}, 85(2):278--287, 2011.

\bibitem{huang2021hybrid}
X.~Huang, Q.~Li, Y.~Tai, Z.~Chen, J.~Zhang, J.~Shi, B.~Gao, and W.~Liu.
\newblock Hybrid deep neural model for hourly solar irradiance forecasting.
\newblock {\em Renewable Energy}, 171:1041--1060, 2021.

\bibitem{kaba2018estimation}
K.~Kaba, M.~Sar{\i}g{\"u}l, M.~Avc{\i}, and H.~M. Kand{\i}rmaz.
\newblock Estimation of daily global solar radiation using deep learning model.
\newblock {\em Energy}, 162:126--135, 2018.

\bibitem{Karandikar:1995vo}
R.~L. Karandikar and V.~G. Kulkarni.
\newblock {Second-order fluid flow models: reflected {B}rownian motion in a
  random environment}.
\newblock {\em Operations Research}, 43(1):77--88, 1995.

\bibitem{latouche2018analysis}
G.~Latouche and G.~T. Nguyen.
\newblock Analysis of fluid flow models.
\newblock {\em Queueing Models and Service Management}, 1(2):1--29, 2018.

\bibitem{latouche2004markov}
G.~Latouche and T.~Takine.
\newblock Markov-renewal fluid queues.
\newblock {\em Journal of Applied Probability}, 41(3):746--757, 2004.

\bibitem{ledoux2008strong}
J.~Ledoux.
\newblock Strong convergence of a class of non-homogeneous {M}arkov arrival
  processes to a {P}oisson process.
\newblock {\em Statistics \& Probability Letters}, 78(4):445--455, 2008.

\bibitem{margolius2016analysis}
B.~Margolius and M.~M. O’Reilly.
\newblock The analysis of cyclic stochastic fluid flows with time-varying
  transition rates.
\newblock {\em Queueing Systems}, 82(1-2):43--73, 2016.

\bibitem{morf2014sunshine}
H.~Morf.
\newblock Sunshine and cloud cover prediction based on {M}arkov processes.
\newblock {\em Solar Energy}, 110:615--626, 2014.

\bibitem{neuts1975probability}
M.~F. Neuts.
\newblock Probability distributions of phase type.
\newblock {\em Liber Amicorum Prof. Emeritus H. Florin}, 1975.

\bibitem{neuts1979versatile}
M.~F. Neuts.
\newblock A versatile {M}arkovian point process.
\newblock {\em Journal of Applied Probability}, 16(4):764--779, 1979.

\bibitem{neuts1994matrix}
M.~F. Neuts.
\newblock {\em Matrix-geometric solutions in stochastic models: an algorithmic
  approach}.
\newblock Courier Corporation, 1994.

\bibitem{neuts1995algorithmic}
M.~F. Neuts.
\newblock {\em Algorithmic probability: a collection of problems}, volume~3.
\newblock CRC Press, 1995.

\bibitem{rodriguez1987some}
I.~Rodriguez-Iturbe, D.~R. Cox, and V.~Isham.
\newblock Some models for rainfall based on stochastic point processes.
\newblock {\em Proceedings of the Royal Society of London. A. Mathematical and
  Physical Sciences}, 410(1839):269--288, 1987.

\bibitem{Rogers:1994uoa}
L.~C.~G. Rogers.
\newblock {Fluid models in queueing theory and Wiener-Hopf factorization of
  Markov chains}.
\newblock {\em The Annals of Applied Probability}, 4(2):390--413, 1994.

\bibitem{ross2010introduction}
S.~M. Ross.
\newblock Introduction to {P}robability {M}odels.
\newblock 2010.

\bibitem{salvadori2001generalized}
G.~Salvadori and C.~De~Michele.
\newblock From generalized pareto to extreme values law: Scaling properties and
  derived features.
\newblock {\em Journal of Geophysical Research: Atmospheres},
  106(D20):24063--24070, 2001.

\bibitem{shrivats2022mean}
A.~V. Shrivats, D.~Firoozi, and S.~Jaimungal.
\newblock A mean-field game approach to equilibrium pricing in solar renewable
  energy certificate markets.
\newblock {\em Mathematical Finance}, 32(3):779--824, 2022.

\bibitem{van2018uniformization}
N.~M. van Dijk, S.~P.~J. van Brummelen, and R.~J. Boucherie.
\newblock Uniformization: {B}asics, extensions and applications.
\newblock {\em Performance Evaluation}, 118:8--32, 2018.

\bibitem{weron2007modeling}
R.~Weron.
\newblock {\em Modeling and forecasting electricity loads and prices: A
  statistical approach}.
\newblock John Wiley \& Sons, 2007.

\bibitem{woo2011impact}
C.-K. Woo, I.~Horowitz, J.~Moore, and A.~Pacheco.
\newblock The impact of wind generation on the electricity spot-market price
  level and variance: {T}he {T}exas experience.
\newblock {\em Energy Policy}, 39(7):3939--3944, 2011.

\end{thebibliography}

\appendix

\section{Proofs}

\subsection{Proof of Lemma~\ref{lem:auxprod3}}\label{sec:lem:auxprod3}

By conditioning on the value of the first jump time $T_1\sim \mathrm{Exp}(\gamma)$,
\begin{align*}
\bm{G}(s,t) = e^{-\gamma(t-s)}\bm{I} + \int_{s}^t \gamma e^{-\gamma (v-s)} \CCs{}(v) \bm{G}(v,t)\dd v
\end{align*}
or equivalently,
\begin{align*}
(e^{\gamma(t-s)}\bm{I})\bm{G}(s,t) = \bm{I} + \int_{s}^t \gamma  \CCs{}(v) \left((e^{\gamma (t-v)}\bm{I})\bm{G}(v,t)\right)\dd v.
\end{align*}
Thus, by (\ref{eq:volterra1}) 
\begin{equation}\label{eq:auxprod3}(e^{\gamma(t-s)}\bm{I})\bm{G}(s,t)= \prod_{v\in (s,t]}(\bm{I}+\gamma\CCs{}(v)\dd v).\end{equation}
Moreover, since $e^{-\gamma(t-s)}\bm{I} = \prod_{v\in (s,t]}(\bm{I}-\gamma\bm{I}\dd v)$, then 
\begin{align*}
\bm{G}(s,t) &= \left(\prod_{v\in (s,t]}(\bm{I}-\gamma\bm{I}\dd v)\right)\left(\prod_{v\in (s,t]}(\bm{I}+\gamma\CCs{}(v)\dd v)\right)\\
& = \prod_{v\in (s,t]}(\bm{I}+\gamma(\CCs{}(v)-\bm{I})\dd v) = \prod_{v\in (s,t]}(\bm{I}+\bm{C}(v)\dd v),
\end{align*}
where the second equality follows from the generalised Trotter formula \cite[Page 126]{gill1994lectures}.

\subsection{Proof of Theorem~\ref{th:multivariateprod1}}\label{sec:th:multivariateprod1}

For $k\in\mathcal{S}$ and $n\ge 1$, let 
\[A_{k,n}:=\big\{ S_1\in [y_1, y_1+\dd y_1),\, S_2-S_1\in [y_2, y_2+\dd y),\, \dots\,,\, S_n-S_{n-1}\in [y_n, y_n+ \dd y_n),\, J(S_n)=k\big\},\]
and $A_{k,0}=\{ J(0)=k\}$. Then, for $n\ge 1$,
\begin{align*}
\mathds{P}(A_{k,n})& =  \sum_{j\in\mathcal{S}}\mathds{E}\left(\mathds{1}\left\{  A_{j,n-1}\right\}\mathds{P}(S_n-S_{n-1} \in [y_n, y_n+ \dd y_n), J(S_n)=k\mid \mathscr{F}_{S_{n-1}})\right)\\
&= \sum_{j\in\mathcal{S}}\mathds{E}\left(\mathds{1}\left\{  A_{j,n-1}\right\} \big(\bm{G}(y_n)\bm{D}(y_n)\big)_{jk}\dd y_n \right)\\
&= \sum_{j\in\mathcal{S}}\mathds{P}( A_{j,n-1}) \big(\bm{G}(y_n)\bm{D}(y_n)\big)_{jk}\dd y_n,
\end{align*}
where the second equality follows from Lemma \ref{lem:auxprod3}. By recursive steps, and taking into account that $\mathds{P}(A_{i,0})=\alpha_i$, we get
\begin{equation}\label{eq:MAPlike1}\mathds{P}(A_{k,n}) = \left(\bm{\alpha} \bm{G}(y_1) \bm{D}(y_1)\bm{G}(y_2) \bm{D}(y_2) \cdots \bm{G}(y_n) \bm{D}(y_n)\right)_k.\end{equation}
Summing (\ref{eq:MAPlike1}) over $k\in\mathcal{S}$ leads to (\ref{eq:MAPlikedensity}).

\subsection{Proof of Proposition~\ref{prop:marginals1}}\label{sec:prop:marginals1}

Integrating (\ref{eq:MAPlikedensity}) w.r.t. $y_1, y_2,\dots, y_{n-1}$ implies that
\[f_{S_n-S_{n-1}}(y_n) = \bm{\alpha} \bm{N}^{n-1} \bm{G}(y_n) \bm{D}(y_n)\bm{1} =   \bm{\alpha} \bm{N}^{n-1} \bm{G}(y_n) (-\bm{C}(y_n)\bm{1}),\]
where the last equality follows from $(\bm{C}(v)+\bm{D}(v))\bm{1} =\bm{0}$, $v\ge 0$. Consequently, $S_n-S_{n-1}\sim \mathrm{IPH}\big(\bm{\alpha}\bm{N}^{n-1}, \{ \bm{C}(v)\}_{v\ge 0}\big)$.

\subsection{Proof of Theorem \ref{th:Psiz3}}\label{sec:th:Psiz3}
For $n\ge 2$, define
\[B_n=\Omega_n\cap \{ F(T_n) - F(0)\le 0\}.\]
Notice that $\{ B_n\}_{n=2}^\infty$ is a partition of $\{ \tau<\infty\}$. Indeed,
\begin{align}
\{ \tau<\infty\} & = \bigcup_{n=2}^\infty \left\{ \bigcap_{k=2}^{n-1} \{ F(T_k)-F(0)>0\}, F(T_n)-F(0)\le 0\right\}\label{eq:auxBn5}\\
& = \bigcup_{n=2}^\infty \left\{ \Omega_{n}, F(T_n)-F(r)\le 0\right\} = \bigcup_{n=2}^\infty B_n;\nonumber
\end{align}
that $\{ B_n\}_{n=2}^\infty$ are disjoint is easily seen from the r.h.s. of (\ref{eq:auxBn5}). Thus, (\ref{eq:appL4}) follows.

\subsection{Proof of Theorem~\ref{th:recursive1}}\label{sec:th:recursive1}

\noindent{\bf Baseline: $n = 2$}. We consider a $2$-bridge.
Fix $i\in\mathcal{S}^+$, $j\in\mathcal{S}^-$. Note that for $a_1,a_2\in\mathds{R}$,
\begin{align*}
& \frac{\partial^2}{\partial a_1\partial a_2}\mathds{E}_{i,z}\Biggr(\exp\left\{-\theta_1\int_{0}^{T_2 }\sigma(J(s))\dd s-\theta_2\sum_{\ell: 0<S_\ell<T_2} k(J(S_\ell) -, J(S_\ell))\right\}\\
 &\qquad\qquad\qquad \times \mathds{1}\left\{\Omega_{2},\, U(T_1-)= U(T_1),\, J(T_2-)=j,\,  T_1 \le a_1,\, T_2-T_1\le a_2\right\}\Biggr)\\
 & \quad = \frac{\partial^2}{\partial a_1\partial a_2}\mathds{E}_{i,z}\Biggr(\exp\left\{-\theta_1\sigma(i)T_1\right\} \mathds{1}\left\{\Omega_{2},\, U(T_1-)= U(T_1),\, J(T_2-)=j,\,  T_1 \le a_1,\, T_2-T_1\le a_2\right\}\Biggr)\\
&\quad = \exp\{-\theta_1 \sigma(i)a_1\}\left(\CCs{+-}\left(z+a_1\right)\right)_{ij}(\gamma e^{-\gamma a_1})(\gamma e^{-\gamma a_2})\mathds{1}_{a_1,a_2\ge 0}\\
 &\quad = \gamma^2 e^{-\theta_1 \sigma(i)a_1-\gamma (a_1+a_2)}\left(\CCs{+-}\left(z+a_1\right)\right)_{ij}\mathds{1}_{a_1,a_2\ge 0}.
\end{align*}
Furthermore, on $\{ J(0)=i, U(0)=z\}\cap\{ \Omega_{2},\, U(T_1-)= U(T_1),\, J(T_2-)=j\}$ we have
 \begin{align*}U(T_2-)& = g_1(T_1,T_2-T_1)=z+T_1+(T_2-T_1),\\F(T_2)-F(0)& = g_2(T_1,T_2-T_1) = r(i)T_1+ r(j)(T_2-T_1),\\
 T_1 & = h_1(U(T_2-), F(T_2)-F(0) )= \frac{-r(j)(U(T_2-)-z) + (F(T_2)-F(0))}{r(i)-r(j)},\\
  T_2-T_1 & = h_2(U(T_2-), F(T_2)-F(0) )= \frac{r(i)(U(T_2-)-z) - (F(T_2)-F(0))}{r(i)-r(j)}.\end{align*} 
  Then, the density transformation implies that
\begin{align}
& \frac{\partial^2}{\partial a_1\partial a_2}\mathds{E}_{i,z}\Biggr(\exp\left\{-\theta_1\int_{0}^{T_2 }\sigma(J(s))\dd s-\theta_2\sum_{\ell: 0<S_\ell<T_2} k(J(S_\ell) -, J(S_\ell))\right\}\nonumber\\
 &\qquad\qquad\qquad \times \mathds{1}\left\{\Omega_{2},\, U(T_1-)= U(T_1),\, J(T_2-)=j,\,  U(T_2-) \le a_1,\, F(T_2)-F(0)\le a_2\right\}\Biggr)\nonumber\\
&\quad = \frac{\gamma^2}{r(i)-r(j)} e^{-\theta_1\sigma(i)a_1-\gamma (a_1-z)}\left(\CCs{+-}\left(z+\tfrac{-r(j)(a_1-z)+a_2}{r(i)-r(j)}\right)\right)_{ij}\mathds{1}_{-r(j)(a_1-z)+a_2\ge 0, r(i)(a_1-z)-y_2\ge 0}.\label{eq:Lambdapartial1}
\end{align}
Similarly,
\begin{align*}
& \frac{\partial^2}{\partial a_1\partial a_2}\mathds{E}_{i,z}\biggr(\exp\biggl\{-\theta_1\int_{0}^{T_2 }\sigma(J(s))\dd s-\theta_2\sum_{\ell: 0<S_\ell<T_2} k(J(S_\ell) -, J(S_\ell))\biggr\}\\
 &\qquad\qquad\qquad \times \mathds{1}\left\{\Omega_{2},\, U(T_1-)\neq U(T_1),\, J(T_2-)=j,\,  T_1 \le a_1,\, T_2-T_1\le a_2\right\}\biggr)\\
 & \quad = \frac{\partial^2}{\partial a_1\partial a_2}\mathds{E}_{i,z}\biggl(\exp\left\{-\theta_1\sigma(i)T_1-\theta_2 k(i,j)\right\} \mathds{1}\left\{\Omega_{2},\, U(T_1-)\neq U(T_1),\right.\\ 
  & \qquad \qquad \qquad \qquad \qquad \left. J(T_2-)=j,\,  T_1 \le a_1,\, T_2-T_1\le a_2\right\}\biggr)\\
&\quad = \exp\{-\theta_1 \sigma(i)a_1-\theta_2 k(i,j)\}\left(\DDs{+-}\left(z+a_1\right)\right)_{ij}(\gamma e^{-\gamma a_1})(\gamma e^{-\gamma a_2})\mathds{1}_{a_1,a_2\ge 0}\\
 &\quad = \gamma^2 e^{-\theta_1 \sigma(i)a_1-\theta_2 k(i,j)-\gamma (a_1+a_2)}\left(\DDs{+-}\left(z+a_1\right)\right)_{ij}\mathds{1}_{a_1,a_2\ge 0}.
\end{align*}
and on $\{ J(0)=i, U(0)=z\}\cap\{ \Omega_{2},\, U(T_1-)\neq U(T_1),\, J(T_2-)=j\}$ we have
 \begin{align*}U(T_2-)& = g^*_1(T_1,T_2-T_1)=z+(T_2-T_1),\\F(T_2)-F(0)& = g_2^*(T_1,T_2-T_1) = r(i)T_1+ r(j)(T_2-T_1),\\
 T_1 & = h_1^*(U(T_2-), F(T_2)-F(0) )= \frac{-r(j)(U(T_2-)-z) + (F(T_2)-F(0))}{r(i)},\\
  T_2-T_1 & = h_2^*(U(T_2-), F(T_2)-F(0) )= U(T_2-)-z.\end{align*} 
This in turn implies that
\begin{align}
& \frac{\partial^2}{\partial a_1\partial a_2}\mathds{E}_{i,z}\Biggr(\exp\left\{-\theta_1\int_{0}^{T_2 }\sigma(J(s))\dd s-\theta_2\sum_{\ell: 0<S_\ell<T_2} k(J(S_\ell) -, J(S_\ell))\right\}\nonumber\\
 &\qquad\qquad\qquad \times \mathds{1}\left\{\Omega_{2},\, U(T_1-)\neq U(T_1),\, J(T_2-)=j,\,  U(T_2-) \le a_1,\, F(T_2)-F(0)\le a_2\right\}\Biggr)\nonumber\\
&\quad = \frac{\gamma^2}{r(i)} e^{-\theta_1 \sigma(i)a_1-\theta_2 k(i,j)-\gamma \frac{(r(i)-r(j))(a_1-z) + a_2}{r(i)}}\left(\DDs{+-}\left(z+\frac{-r(j)(a_1-z) + a_2}{r(i)}\right)\right)_{ij}\nonumber \\
&\qquad  \qquad \qquad \times \mathds{1}_{-r(j)(a_1-z) + a_2\ge 0,\, a_1-z\ge 0}.\label{eq:Lambdapartial2}
\end{align}
Equation (\ref{eq:Lambdamain1}) readily follows from (\ref{eq:Lambdapartial1}) and (\ref{eq:Lambdapartial2}).

\medskip
\noindent{\bf Induction: $n\ge 3$}. For $w\in\{ 1,\dots, n-1\}$ define the disjoint events
\[E_{w}=\left\{ F(T_w) < \min_{m\in\{ 1,\dots, n-1\}\setminus\{ w\}} F(T_m)\right\},\]
which correspond to the cases where the intermediate minimum is reached at the $w$-th jump.
Clearly $ \{ E_{w}\}_{w\in\{ 1,\dots, n-1\}}$ is a partition of the sample space. Given this, (\ref{eq:Lambdamain2}) will readily follow once we show that 
\begin{align}
\left(\bm{\Gamma}^{(n,w,z)} (s,\ell)\right)_{ij}\,\dd s\,\dd \ell &= \mathds{E}_{i,z}\Biggr(\exp\left\{-\theta_1\int_{0}^{T_n }\sigma(J(s))\dd s-\theta_2\sum_{\ell: 0<S_\ell<T_n} k(J(S_\ell) -, J(S_\ell))\right\}\nonumber\\
 & \times \mathds{1}\left\{\Omega_{n}, E_w,\,J(T_{n}-)=j,\, U(T_{n}-)\in [s, s+\dd s),\, F(T_{n})-F(0)\in [\ell,\ell +\dd \ell)\right\}\Biggr),\label{eq:defGammapartial1}
\end{align}
for $\bm{\Gamma}^{(n,w,z)}(s,\ell)$ of the form (\ref{eq:Gamma1}), (\ref{eq:Gamma2}) and (\ref{eq:Gamma3}), which we do next. We proceed by assuming that $\bm{\Gamma}^{(n,w,z)}(s,\ell)$ is defined by (\ref{eq:defGammapartial1}): from this we will verify that this expression matches (\ref{eq:Gamma1}), (\ref{eq:Gamma2}) and (\ref{eq:Gamma3}) for $w\in\{ 1,\dots,n\}$.
\begin{itemize}
  \item {\bf Case $w=1$.}
  On $\{ J(0)=i, U(0)=z\}\cap\Omega_n\cap E_1$, the first jump of $J$ is to $\mathcal{S}^+$, $F(T_1)-F(0)=r(i) T_1$, and $F(T_n) < F(T_1) < \min\{ F(T_2), \cdots , F(T_{n-1})\}$; see Figure \ref{fig:Lambdanw1}.
  Thus, integrating w.r.t. the density function of $T_1$, splitting into the cases in which $U(T_1)= U(T_1-)$ and $U(T_1)\neq U(T_1-)$, and employing the strong Markov property,
  \begin{align} \left(\bm{\Gamma}^{(n,1,z)}(s,\ell)\right)_{ij} & = \sum_{i'\in\mathcal{S}}\int_0^\infty e^{-\theta_1 \sigma(i) u}\left(\gamma e^{-\gamma u}\bar{c}^{++}_{ii'}(u)\right) \left(\bm{\Lambda}^{(n-1,z+u)}(s,\ell-r(i) u\bm{e})\right)_{i'j}\dd u\nonumber\\
  & \quad+\sum_{i'\in\mathcal{S}}e^{-\theta_2 k(i,i')}\int_0^\infty e^{-\theta_1 \sigma(i)u}\left(\gamma e^{-\gamma u}\bar{d}^{++}_{ii'}(u)\right) \left(\bm{\Lambda}^{(n-1,0)}(s,\ell-r(i) u\bm{e})\right)_{i'j}\dd u,\nonumber\label{eq:AOmega1}
  \end{align}
  from which (\ref{eq:Gamma1}) follows.
  \begin{figure}[h]
  \centering
  \includegraphics[scale=1.7]{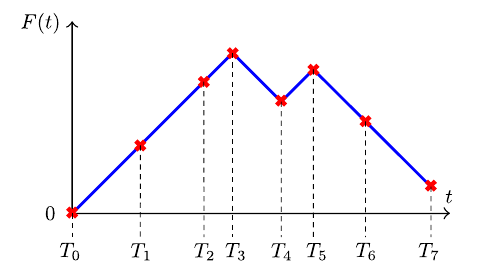}
\caption{A sample path contained in $E_{1}\cap \Omega_{n}$ for $n=7$ where each red cross represents a Poisson time $T_w$. Notice that the condition $F(T_n) < F(T_1) < \min\{ F(T_2), \cdots , F(T_{n-1})\}$ indeed holds. \label{fig:Lambdanw1}}
\end{figure}

\item {\bf Case $w\in\{ 2,\dots,n-2\}$.} On $\{ J(0)=i, U(0)=z\}\cap\Omega_n\cap E_w$, we have $F(T_0) < F(T_w) < \min\{ F(T_2), \cdots , F(T_{w-1})\}$ and $F(T_w) < F(T_n) < \min\{ F(T_{w+1}), \cdots , F(T_{n-1})\}$; see Figure \ref{fig:Lambdanwmid}. Then,
\begin{align*}
&\left(\bm{\Gamma}^{(n,w,z)}(s,\ell)\right)_{ij}= \sum_{j'\in\mathcal{S}^-}\sum_{i'\in\mathcal{S}^+} \int_0^\infty \int_{v=0\vee \ell}^\infty \left(\bm{\Lambda}^{(w,z)}(u, v)\right)_{ij'} \bar{c}^{-+}_{j'i'}(u) \left(\bm{\Lambda}^{(n-w,u)}(s, \ell-v)\right)_{i'j}\dd u \dd v\\
&\quad \quad + \sum_{j'\in\mathcal{S}^-}\sum_{i'\in\mathcal{S}^+} e^{-\theta_2 k(j',i')}\int_0^\infty \int_{v=0\vee \ell}^\infty \left(\bm{\Lambda}^{(w,z)}(u, v)\right)_{ij'} \bar{d}^{-+}_{j'i'}(u) \left(\bm{\Lambda}^{(n-w,0)}(s, \ell-v)\right)_{i'j}\dd u \dd v,
\end{align*}
and (\ref{eq:Gamma2}) follows.
 \begin{figure}[h]
  \centering
    \includegraphics[scale=1.7]{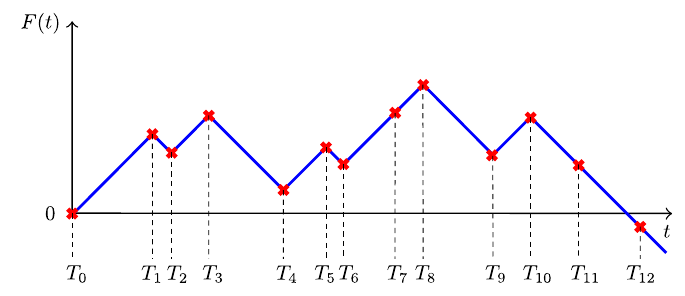}
\caption{A sample path contained in $E_{w}\cap \Omega_{n}$ for $w=4$ and $n=12$. Each red cross represents a Poisson time $T_\ell$. Notice that both conditions $F(T_0) < F(T_w) < \min\{ F(T_2), \cdots , F(T_{w-1})\}$ and $F(T_w) < F(T_n) < \min\{ F(T_{w+1}), \cdots , F(T_{n-1})\}$ hold. \label{fig:Lambdanwmid}}
\end{figure}
 \item {\bf Case $w=n-1$.} On $\{ J(0)=i, U(0)=z, J(T_n-)=j\}\cap\Omega_n\cap E_{n-1}$, we have $F(T_0) < F(T_{n-1}) < \min\{ F(T_2), \cdots , F(T_{n-2})\}$, $J(T_{n-1})\in\mathcal{S}^-$, and $F(T_n)-F(T_{n-1})=r(j) (T_n-T_{n-1})$; see Figure \ref{fig:Lambdanwfin}. Integrating w.r.t. the density function of $T_n-T_{n-1}$, splitting into the cases in which $U(T_{n-1})= U(T_{n-1}-)$ and $U(T_{n-1})\neq U(T_{n-1}-)$, and employing the strong Markov property,
  \begin{align} \left(\bm{\Gamma}^{(n,n-1,z)}(s,\ell)\right)_{ij} & = \sum_{j'\in\mathcal{S}^-}\int_0^s  \left(\bm{\Lambda}^{(n-1,z)}(s-u,\ell-r(j) u)\right)_{ij'} \left(\gamma e^{-\gamma u}\bar{c}^{--}_{j'j}(u)\right) \dd u\nonumber\\
  & \quad+\sum_{j'\in\mathcal{S}^-}e^{-\theta_2 k(j',j)}\int_0^\infty \left(\bm{\Lambda}^{(n-1,z)}(u,\ell-r(j) s)\right)_{ij'} \left(\bar{d}^{--}_{j'j}(u) \gamma e^{-\gamma s}\right)\dd u,\nonumber
  \end{align}
so that (\ref{eq:Gamma3}) follows and the proof is concluded.
\begin{figure}[h]
  \centering
  \includegraphics[scale=1.7]{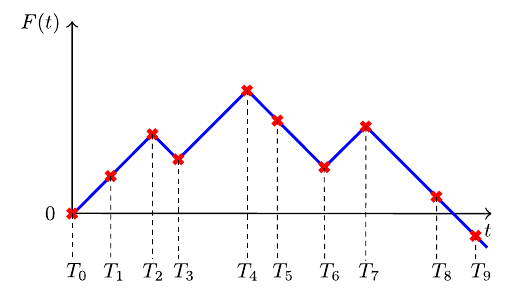}
\caption{A sample path contained in $E_{n-1}\cap \Omega_{n}$ for $n=9$. Each red cross represents a Poisson time $T_w$. Note that $F(T_0) < F(T_{n-1}) < \min\{ F(T_2), \cdots , F(T_{n-2})\}$. \label{fig:Lambdanwfin}}
\end{figure}
  \end{itemize}

  \subsection{Proof of Theorem~\ref{th:timeinhomogeneouspsi1}}\label{sec:th:timeinhomogeneouspsi1}

Fix $t>0$ and let 
\[B_n^*=\Omega_n\cap \{ T_n \le t,\, F(T_n) - F(0)\le 0\}.\]
Then, similarly to (\ref{eq:auxBn5}), we can argue that $\{ B^*_n\}_{n\ge 2}$ is a partition of $\{ \tau_*\le t\}$. Since $U(T_n-)=T_n + z$, then (\ref{eq:Psistar3}) follows. Moreover,
\begin{align*}
\sum_{n=m}^\infty \big(\bm{L}^{(n,z)}(0,0,t+z, 0)\big)_{ij}&\le \mathds{P}\big(T_m\le t\big)\\
&=O\left((\gamma t)^{m}/m!\right)\\
& = O\left(e^{m\log(\gamma t) - (m+1/2)\log m + m}\right)\\
& = o\left(e^{-m(\log m - \log(\gamma t)-1)}\right),
\end{align*}
where the second equality follows from  \cite[Corollary 1.(ii)]{glynn1987upper} and the third equality follows from Stirling's formula; see e.g., \cite[p. 52]{feller1968introduction}.

\section{Alternative construction of a DMArP}\label{sec:altconst}

We devote this section to constructing the DMArP $N$, as well as its underlying jump process $J$, without the uniform boundedness condition of Assumption \ref{ass:bounded1}. Our construction employs the theory of \emph{piecewise deterministic Markov processes} (PDMPs), introduced in \cite{davis1984piecewise}, for which we first provide a brief summary.

A piecewise deterministic Markov process $X=\{ X(t)\}_{t\ge 0}$ is a stochastic process whose state-space $\mathcal{E}$ is a union of subsets $\mathcal{E}_j$ of Euclidean spaces, and whose evolution is characterized by three components:
\begin{itemize}
\item a \emph{vector field} $\mathfrak{X}$ acting on $\mathcal{E}$,
\item a \emph{jump rate} $\lambda:\mathcal{E}\rightarrow \mathds{R}_+$, 
\item a \emph{transition probability measure} $Q:\mathds{B}_{\mathcal{E}}\times\mathcal{E}\rightarrow \mathds{R}_+$, where $\mathcal{B}_{\mathcal{E}}$ is shorthand for the union of the Borel $\sigma$-algebras associated to each $\mathcal{E}_j$.
\end{itemize}

The triple $(\mathfrak{X}, \lambda, Q)$ is referred to as the \emph{local characteristics} of the PDMP and dictate the behavior of $X$ as follows. Here, we will use PDMPs such that the solution path of the vector field $\mathfrak{X}$ is contained in $\mathcal{E}$ for any initial point $x_0\in\mathcal{E}$, or in other words, $\mathfrak{X}$ does not point towards boundary points of $\mathcal{E}$. Suppose that $X(0)\in \mathcal{E}$ is arbitrary but fixed; then the process $X$ moves deterministically in $\mathcal{E}$ according to $\mathfrak{X}$ up to its first jump epoch, the latter which occurs at time $t\ge 0$ with intensity $\lambda \big(X({t-})\big)$. Say that the first jump occurs at time $\beta_1>0$; then, $X(\beta_1)$ jumps to some state in $\mathcal{E}$ according to the probability measure $Q\big(X(\beta_1-),\cdot \big)$, from which we repeat the aforementioned scheme. By recursively replicating these steps, we arrive at a sequence of jump times $0 < \beta_1<\beta_2<\beta_3<\dots$ and points $X(\beta_{1}), X(\beta_{2}), X(\beta_{3})$ which define the process $X$ up to time $\beta_*:=\lim_{n\rightarrow\infty} \beta_n$. It can be shown \cite{davis2018markov} that the PDMP $X$ as constructed above is a strong Markov process characterized by the infinitesimal generator

\begin{equation}\label{eq:generatorMarkov1}\mathscr{L} f (x) = \mathfrak{X}f(x) + \lambda(x) \int_{\mathcal{E}} (f(y)-f(x))\,Q(x;\dd y).\end{equation}

For our DMArP construction, we let $\mathcal{E}=\mathds{R}_+\times\mathcal{S}$ where $\mathcal{S}$. Moreover, we take the local characteristics $(\mathfrak{X},\lambda,Q)$:
\begin{enumerate}
\item[(I.1)]\label{it:local1} $\mathfrak{X}f(v,i)=\frac{\partial f}{\partial v}(v,i)$ for $i\in\mathcal{S},\, v \ge 0$ and $ f\in C^{1}( \mathds{R}_+\times\mathcal{S})$, where the partial derivative at $v=0$ is understood as a partial derivative from the right, and $ C^{1}(\mathds{R}_+\times\mathcal{S})$ is the space of functions on $\mathds{R}_+\times\mathcal{S}$ with a continuous partial derivative on their first entry.
\item[(I.2)]\label{it:local2} $\lambda(v,i)= c_{i}(v)$.
\item[(I.3)]\label{it:local3} For each $v\ge 0$, and $i,j\in\mathcal{S}$, $u\ge 0$,
\[Q(v,i; \dd u,\dd j)= \left\{  \begin{array}{ccc} \frac{c_{ij}(v)}{c_i(v)} & \mbox{if} & i \neq j, u=v, c_i(v)>0\\ \frac{d_{ij}(v)}{c_i(v)} & \mbox{if} & n u=0, c_i(v)>0\\
\mathds{1}{\{i=j, v=u\}} & \mbox{if} & c_i(v)=0.\end{array}\right.\]
\end{enumerate}

A PDMP $X$ with local characteristics (I.1), (I.2), and (I.3) has the following behavior. While in $\mathds{R}_+\times {i}$, it jumps at a rate of $c_i(\cdot)$ at times $\beta_1, \beta_2, \dots$. Suppose that $X(\beta_{n})=(v,i)$ for some $v\ge 0$ and $i\in\mathcal{S}$. Then, $X(\beta_{n}+s)=(v+s, i)$ for all $s\in [0, \beta_{n+1}-\beta_{n})$. Also, at time $\beta_n+s'$, where $s'=\beta_{n+1}-\beta_n$, $X(\beta_n+s')$ can be either $(v+s', j)$ with probability $c_{ij}(v+s')/c_i(v+s')$ or $(0, j)$ with probability $d_{ij}(v+s')/c_i(v+s')$. Multiplying these probabilities by the jump intensity $c_i(v+s')$, we find that after $\beta_n$ with $X(\beta_{n})=(v,i)$, the next jump time $\beta_{n+1}$ of the second coordinate process of $X$ occurs $s'$ units of time later according to the intensity matrix $\bm{C}(v+s')+\bm{D}(v+s')$. At the same time, the first coordinate process jumps to $0$ at this epoch only due to the mass contribution from $\bm{D}(v+s')$ (otherwise it does not jump). Consequently, if we let $J$ be the second coordinate process of $X$ and $N$ be the counting process associated with the visits of the first coordinate process to ${0}\times\mathcal{S}$ in the time interval $(0,\infty)$, then we can see that the duration process $U$ coincides with the first coordinate process of $X$. Furthermore, this construction agrees with the distributional properties (\ref{eq:defMArPR1}) and (\ref{eq:defMArPR2}) that characterize a DMArP.

According to (\ref{eq:generatorMarkov1}), the infinitesimal generator $\mathscr{L}$ of the PDMP with local characteristics (I.1), (I.2) and (I.3) is given by
\begin{equation*}\mathscr{L}f(v,i)=\frac{\partial f}{\partial v} (v,i) + \sum_{j\in \mathcal{S}} [f(v,i)-f(v,i)]\, C_{ij}(v) + \sum_{j\in \mathcal{S}} [f(v,j)-f(v,i)]\, D_{ij}(s),\quad i\in\mathcal{S},\,s\in\mathds{R}_+,\end{equation*}
which ultimately corresponds to that of the bivariate process $(U,J)$. Note that this construction defines the processes $J$, $U$, and $N$ up to $\beta_*$. Below we present a condition that guarantees that $\beta_*=+\infty$, and thus, these processes are defined in $\mathds{R}_+$.

\begin{theorem}
Suppose that $\int_0^t c_i(a)\dd a<\infty$ for all $i\in\mathcal{S}, t\ge 0$. Then, $\beta_*=\infty$ almost surely.
\end{theorem}

\begin{proof}
For $X(0)=(v_0,i_0)$, assume that there exists some $T>0$ such that the event $\{\beta_*< T\}=\{\beta_\ell \le T \mbox{ for all }\ell \ge 1\}$ occurs with strictly  positive probability. By construction, on the event $\{\beta_\ell\le T\}$, the first coordinate of $X(\beta_\ell)$ cannot be larger than $T+v_0$. Then, for $b_\ell\le T$, $v_\ell\le T+v_0$ and $i_\ell\in\mathcal{S}$,
\begin{align}
\mathds{P}(\beta_{\ell+1}\le T | \beta_\ell= b_\ell, X(\beta_\ell)=(v_\ell,i_\ell)) &  = 1-e^{-\int_{\beta_\ell}^{T} c_{i_\ell}(a+v_\ell)\dd a} = 1-e^{-\int_{\beta_\ell+v_\ell}^{T+v_\ell} c_{i_\ell}(a)\dd a}\nonumber\\
& \le 1 - \exp\left(-\sum_{j\in\mathcal{S}}\int_{0}^{2T+v_0} c_{j}(a)\dd a\right)=:A.
\end{align}
Note that $A$ does not depend on $b_\ell, v_\ell$ or $i_\ell$, so that we can write
\[\mathds{P}(\beta_{\ell+1}\le T | \beta_\ell\le T)\le A.\]
Furthermore, since $\int_0^t c_i(a)\dd a<\infty$ for all $t\ge 0$, then $A<1$. Consequently, by recursive conditioning we get
\[\mathds{P}(\beta_\ell \le T \mbox{ for all }\ell \ge 1)= \lim_{\ell \rightarrow\infty} \mathds{P}(\beta_\ell \le T) \le \lim_{\ell \rightarrow\infty} A^{\ell} = 0, \]
which contradicts the initial assumption. Thus, $\beta_*=\infty$ holds almost surely.
\end{proof}


\end{document}